%% file: vecLapmg.tex
\documentclass[10pt]{amsart}
\input{./mysetting.tex}

\begin{document}
\title{MultiGrid Preconditioners for Mixed Finite Element Methods of Vector Laplacian}
\author{Long Chen, Yongke Wu, Lin Zhong, and Jie Zhou}\date{\today}
\thanks{L. Chen was supported by NSF Grant DMS-1418934. Y. Wu was supported by the National Natural Science Foundation of China (11501088) and partially supported by NSF Grant DMS-1115961. L. Zhong was supported by NSF Grant DMS-1115961 and DMS-1418934. J. Zhou was supported by  doctoral research project of Xiangtan University (09kzkz08050).}

\address[L.~Chen]{Department of Mathematics, University of California at Irvine, Irvine, CA 92697, USA}
\email{chenlong@math.uci.edu}

\address[Y.~Wu]{School of Mathematical Sciences, University of Electronic Science and Technology of China, Chengdu 611731, China.}
\email{wuyongke1982@sina.com}

\address[L.~Zhong]{Department of Mathematics, University of California at Irvine, Irvine, CA 92697, USA}
\email{lzhong1@uci.edu}

\address[J.~Zhou]{School of Mathematical and Computational Sciences, Xiangtan University, Xiangtan, 411105, China}
\email{xnuzj2004@163.com}


\begin{abstract}
Due to the indefiniteness and poor spectral properties,
the discretized linear algebraic system of the vector Laplacian by mixed finite element methods is hard to solve. A block diagonal preconditioner has been developed and shown to be an effective preconditioner by Arnold, Falk, and Winther [Acta Numerica, 15:1--155, 2006]. The purpose of this paper is to propose alternative and effective block diagonal and block triangular preconditioners for solving this saddle point system. A variable V-cycle multigrid method with the standard point-wise Gauss-Seidel smoother is proved to be a good preconditioner for a discrete vector Laplacian operator. This multigrid solver will be further used to build preconditioners for the saddle point systems of the vector Laplacian and the Maxwell equations with divergent free constraint. The major benefit of our approach is that the point-wise Gauss-Seidel smoother is more algebraic and can be easily implemented as a black-box smoother. 
\end{abstract}

\maketitle

\section{Introduction}

Discretization of the vector Laplacian in spaces $\bs H_{0}(\curl)$ and $\bs H_{0}(\div)$ by mixed finite element methods is well-studied in~\cite{Arnold2006}. The discretized linear algebraic system is ill-conditioned and in the saddle point form which leads to the slow convergence of classical iterative methods as the size of the system becomes large. In~\cite{Arnold2006}, a block diagonal preconditioner has been developed and shown to be an effective preconditioner. The purpose of this paper is to present alternative and effective block diagonal and block triangular preconditioners for solving these saddle point systems.

Due to the similarity of the problems arising from spaces $\bs H_{0}(\curl)$ and $\bs H_{0}(\div)$, we use the mixed formulation of the vector Laplacian in $\bs H_{0}(\curl)$ as an example to illustrate our approach. Choosing appropriate finite element spaces $S_{h} \subset H_{0}^{1}$ (a vertex element space) and $ \bs U_{h} \subset \bs H_{0}(\curl)$ (an edge element space), the mixed formulation is: Find $\sigma_h \in S_{h}, \bs u_h\in \bs U_{h}$ such that
\begin{align*}
\left\{
\begin{aligned}
-(\sigma_h, \tau_h) + (\bs u_h, \grad \tau_h) &= 0 & \text{ for all } \tau_h \in S_{h},\\
(\grad \sigma_h, \bs v_h) + (\curl\bs  u_h, \curl\bs v_h) &=( \bs f, \bs v_{h}) & \text{ for all }\bs v_h\in \bs U_{h}.
\end{aligned}
\right.
\end{align*}
The corresponding matrix formulation is
\begin{equation}\label{matrix-form}
\begin{pmatrix}
-M_v & B\\
B^T & C^TM_fC
\end{pmatrix}
\begin{pmatrix}
\sigma_h\\
\bs u_h
\end{pmatrix}
=
\begin{pmatrix}
0\\
\bs f
\end{pmatrix}.
\end{equation}
Here $M_v$ and $M_f$ are mass matrices of the vertex element and the face element, respectively, $B^T$ corresponds to a scaled $\grad$ operator, and $C$ corresponds to the $\curl$ operator.

Based on the stability of \eqref{matrix-form} in $H_0^1\times \bs H_0(\curl)$ norm, in~\cite{Arnold2006}, a block diagonal preconditioner in the form
\begin{equation*}
\begin{pmatrix}
(I+ G^T M_e G)^{-1} &O \\
O& (I+ C^TM_fC)^{-1}
\end{pmatrix},
\end{equation*}
with $G = M_e^{-1}B^T$,
 is proposed and the preconditioned Krylov space method is shown to converge with optimal complexity. To compute the inverse operators in the diagonal, multigrid methods based on additive or multiplicative overlapping Schwarz smoothers ~\cite{Arnold.D;Falk.R;Winther.R2000}, multigrid methods based on Hiptimair smoothers \cite{Hiptmair.R1997,Hiptmair.R1999}, or HX auxiliary space preconditioner \cite{Hiptmair.R;Xu.J2007} can be used. In all these methods,  to achieve a mesh independent condition number, a special smoother taking care of the large kernel of the $\curl$ (or $\div$) differential operators is needed. 

In contrast, we shall apply multigrid methods with the standard point-wise  Gauss-Seidel (G-S) smoother to the Schur complement of the $(1,1)$ block
\begin{equation}\label{intro:Schur}
A = B^TM^{-1}_vB + C^TM_fC
\end{equation}
which is a matrix representation of the following identity of the vector Laplacian
 \begin{equation*}\label{vecLap}
- \Delta \bs u =  -\grad \div\bs u + \curl \curl \bs u.
\end{equation*}
In \eqref{intro:Schur}, the inverse of the mass matrix, i.e., $M^{-1}_v$ is dense. To be practical, the exact Schur complement can be replaced by an approximation
$$
\tilde A = B^T\tilde M^{-1}_vB + C^TM_fC,
$$
with $\tilde M_v$ an easy-to-invert matrix, e.g., the diagonal or a mass lumping of $M_v$.

We shall prove that a variable V-cycle multigrid method using the standard point-wise Gauss-Seidel smoother is a good preconditioner for the Schur complement $A$ or its approximation $\tilde A$. The major benefit of our approach is that the point-wise Gauss-Seidel smoother is more algebraic and can be easily implemented as a black-box smoother. The block smoothers proposed in~\cite{Arnold.D;Falk.R;Winther.R2000} for the $\bs H(\curl)$ and $\bs H(\div)$ problems, however, requires more geometric information and solving local problems in small patches.

Although the finite element spaces are nested and $A$ is symmetric positive definite, due to the inverse of the mass matrix, the bilinear forms in the coarse grid are non-inherited from the fine one. To overcome this difficulty, we shall follow the multigrid framework developed by Bramble, Pasciak, and Xu~\cite{Bramble.J;Pasciak.J;Xu.J1991}. In this framework, we need only to verify two conditions: (1) Regularity and approximation assumption; (2) Smoothing property. Since $A$ is symmetric and positive definite, the smoothing property of the Gauss-Seidel smoother is well known, see e.g.~\cite{Bramble.J;Pasciak.J1992}.
To prove the approximation property, we make use of the $L^2$-error estimates of mixed finite element methods established in~\cite{Arnold.D;Falk.R;Winther.R2000} and thus have to assume the full regularity of elliptic equations. Numerically our method works well for the case when the full regularity does not hold. With the approximation and smoothing properties, we show that one V-cycle is an effective preconditioner. As noticed in \cite{Bramble.J;Pasciak.J1992}, W-cycle or two V-cycles may not be a valid preconditioner as the corresponding operator may not be positive definite. In other words, the proposed multigrid method for the Schur complement cannot be used as an iterative method but one V-cycle can be used as an effective preconditioner.

The multigrid preconditioner for $\tilde A$ will be used to build preconditioners for \eqref{matrix-form}. We propose a block diagonal preconditioner and a block triangular preconditioner:
\begin{equation}\label{preconditioner}
\begin{pmatrix}
M_v^{-1} &O \\
O& \tilde A^{-1}
\end{pmatrix},
\quad \text{and }
\begin{pmatrix}
I   &  \tilde M_v^{-1} B \\
0  &   I
 \end{pmatrix}
\begin{pmatrix}
-\tilde M_v    &   0  \\
B^T   &  \tilde A
 \end{pmatrix}
^{-1}.
\end{equation}
The action $M_{v}^{-1}$ can be further approximated by $\tilde M_v^{-1}$ and $\tilde A^{-1}$ by one V-cycle multigrid.
Following the framework of \cite{MardalWinther2004}, we prove that the preconditioned system using these two preconditioners has a uniformly bounded conditional number by establishing a new stability result of the saddle point system \eqref{matrix-form} in the $\|\cdot \|\times \|\cdot\|_A$ norm.

As an application we further consider a prototype of Maxwell equations with divergence-free constraint
$$
\curl \curl \bs u = \bs f, \; \div \bs u = 0, \; \text{ in } \Omega.
$$
A regularized system obtained by the augmented Lagrangian method~\cite{Fortin.Michel;Glowinski.Roland1983} has the form
\begin{equation}\label{MaxABB}
\begin{pmatrix}
A & B^T\\
B & O
\end{pmatrix},
\end{equation}
where $A$ is the vector Laplacian in $\bs H_0(\curl)$. We then construct a block diagonal preconditioner and a block triangular preconditioner
\begin{equation}
\begin{pmatrix}
A^{-1} &0 \\
0&  M_v^{-1}
\end{pmatrix},
\quad
\text{ and }
\,
\begin{pmatrix}
I & G\\
O & -\tilde M_v^{-1}A_p
\end{pmatrix}
\begin{pmatrix}
\tilde A & O\\
B & A_p
\end{pmatrix}^{-1},
\end{equation}
and prove that they are uniformly bounded preconditioners for the Maxwell system \eqref{MaxABB}. Our preconditioners are new and different with the  solver proposed in~\cite{ChenXuZou2010}.
%

The paper is organized as follows. In Section 2, we introduce the discretization of the mixed formulation of the vector Laplacian, and prove stability results. In Section 3, we consider the multigrid methods for the discrete vector Laplacian and verify the approximation and smoothing properties. In Section 4, we propose the uniform preconditioner for the vector Laplacian and apply to Maxwell equation in the saddle point form. At last, we support our theoretical results with numerical experiments.

\section{Discretization}
In this section, we first recall the function spaces and finite element spaces, and then present discrete formulations of the vector Laplacian problems in both space $\bs H_{0}(\curl)$ and space $\bs H_{0}(\div)$. We shall define a new norm using the Schur complement and prove corresponding Poincar\'e inequalities and inverse inequalities.

We assume that $\Omega$ is a bounded and convex polyhedron in $\mbb R^3$ with a simple topology (homomorphism to a ball), and it is triangulated into a mesh $\mcal T_h$ with size $h$.  We assume that the mesh $\mcal T_h$ belongs to a shape regular and quasi-uniform family.

\subsection{Function Spaces and Finite Element Spaces}We use $L^2(\Omega)$ to denote the space of all square integrable scalar or vector functions on $\Omega$ and $(\cdot,\cdot)$ for both the scalar and vector $L^2$-inner product.
Given a differential operator $\mcal D = \grad, \curl,$ or $\div$, we introduce the Sobolev space $H(\mcal D,\Omega) = \{v\in L^2(\Omega), \mcal D v\in L^2(\Omega)\}$. For $\mcal D=\grad$, $H(\grad,\Omega)$ is the standard $H^1(\Omega)$. For simplicity, we will suppress the domain $\Omega$ in the notation. Let $\bs n$ be the unit outwards normal vector of $\partial \Omega$. We further introduce the following Sobolev spaces on domain $\Omega$ with homogenous traces:
\begin{eqnarray*}
H_0^1 & = & \{ u \in  H^1:  u = 0 \hbox{~~on~} {\partial \Omega}\},\\
\bs H_0(\curl) & = & \{\bs u \in \bs H(\curl): \bs u\times \bs n = 0 \hbox{~~on~} {\partial \Omega}\},\\
\bs H_0(\div) & = & \{\bs u \in \bs H(\div): \bs u\cdot \bs n =0 \hbox{~~on~} {\partial \Omega}\},\\
\text{ and } \quad L_0^2 & = & \{u \in L^2: \int_{\Omega} u ~dx= 0\}.{\large 
}\end{eqnarray*}
Then, let us recall the following finite element spaces:
\begin{itemize}
\item $S_h\subset H^1_0$ is the well-known Lagrange elements, i.e., continuous and piecewise polynomials,
\item $\bs U_h\subset \bs H_0(\curl)$ is the edge element space~\cite{Nedelec.J1980,Nedelec.J1986},
\item  $\bs V_h\subset \bs H_0(\div)$ is the face element space~\cite{Raviart.P;Thomas.J1977,Nedelec.J1980,Brezzi.F;Douglas.J;Marini.L1985,Nedelec.J1986,Brezzi.F;Douglas.J;Duran.R;Fortin.M1987,Brezzi.F;Fortin.M1991},
\item $W_h\subset L^2_0$ is discontinuous and piecewise polynomial space.
\end{itemize}

To discretize the vector Laplacian problem posed in $\bs H_{0}(\div)$ or $\bs H_{0}(\curl)$, we start from the following de Rham complex
$$0{\longrightarrow}H^1_0 \stackrel{\grad}{\longrightarrow} \bs H_0(\curl) \stackrel{\curl}{\longrightarrow} \bs H_0(\div) \stackrel{\div}{\longrightarrow} L^2_0{\longrightarrow}0.$$
We choose appropriate degrees and types of finite element spaces such that the discrete de Rham complex holds
\begin{equation}\label{exact}
0{\longrightarrow}S_h \stackrel{\grad}{\longrightarrow} \bs U_h \stackrel{\curl}{\longrightarrow} \bs V_h \stackrel{\div}{\longrightarrow} W_h{\longrightarrow}0.
\end{equation}
Important examples are: $S_{h}$ is the linear Lagrange element; $\bs U_{h}$ is the lowest order Nedelec edge element; $\bs V_{h}$ is the lowest order Raviart-Thomas element, and  $W_{h}$ is the piecewise constant.

We now define weak differential operators and introduce the following exact sequence in the reversed ordering:
\begin{equation}\label{weakexact}
0{\longleftarrow}S_h \stackrel{\div_{h}}{\longleftarrow} \bs U_h \stackrel{\curl_{h}}{\longleftarrow} \bs V_h \stackrel{\grad_{h}}{\longleftarrow} W_h{\longleftarrow}0.
\end{equation}
The weak divergence $\div_{h}: \bs U_{h} \to S_h$ is defined as the adjoint of $-\grad$ operator in the $L^2$-inner product, i.e., $\div_{h}\bs  w_h \in S_h$, s.t.,
\begin{equation}\label{weakdiv}
(\div_{h} \bs w_h, v_h) : = - (\bs w_h, \grad v_h)\quad \text{ for all } v_h\in S_h.
\end{equation}
Weak $\curl$ operator $\curl_{h}$ and weak $\grad$ operator $\grad_{h}$ are defined similarly.
For a given $\bs w_{h} \in \bs V_{h}$, define $\curl_{h}\bs w_{h}\in \bs U_{h}$ as
\begin{equation}\label{weakcurl}
(\curl_{h}\bs  w_h,\bs v_h) : =  (\bs w_h, \curl\bs v_h)\quad \text{ for all }\bs v_h\in \bs U_h.
\end{equation}
For a given $w_{h} \in \bs W_{h}$, define $\grad_{h}w_{h}\in \bs V_{h}$ as
\begin{equation}\label{weakgrad}
(\grad_{h}  w_h, \bs v_h) : = - ( w_h, \div\bs  v_h)\quad \text{ for all } \bs v_h\in \bs V_h.
\end{equation}
In the limiting case when $h\to 0$, these weak differential operators becomes the so-called co-differential operators, c.f. \cite{Arnold2006}, and will be denoted by $\mcal D^w$. 

The exactness of~\eqref{weakexact} can be easily verified by the definition and the exactness of~\eqref{exact}. Note that the inverse of mass matrices will be involved when computing the weak differential operators and thus they are global operators.


We introduce the null space of differential operators:
$$
Z_h^c = \bs U_h \cap \ker(\curl), \quad \text{ and } \, Z_h^d = \bs V_h \cap \ker(\div),
$$
and the null space of weak differential operators 
$$
K_h^c = \bs U_h \cap \ker(\div_h), \quad \text{ and } \, K_h^d = \bs V_h \cap \ker(\curl_h). 
$$
Similar notation $Z^c, Z^d, K^c, K^d$ will be used for the null spaces in the continuous level when the subscript $h$ is skipped. 

According to the exact sequence \eqref{exact}, we have 
the discrete Hodge decompositions \cite{Arnold2006}:
\begin{align*}
\bs U_{h} &= Z_h^c \oplus^{\bot} K_h^c = \grad S_{h} \oplus^{\bot} \curl_{h} \bs V_{h},\\
\bs V_{h} &= Z_h^d \oplus^{\bot} K_h^d = \curl \bs U_{h} \oplus^{\bot} \grad_{h}  W_h.
\end{align*}
The notation $\oplus^{\bot}$ stands for the $L^2$ orthogonal decomposition. These discrete version of Hodge decompositions play an important role in the  analysis.

We update the exact sequences as:
\begin{equation}\label{exactZK}
0{\longrightarrow}S_h \stackrel{\grad}{\longrightarrow} Z_h^c\oplus K_h^c \stackrel{\curl}{\longrightarrow} Z_h^d\oplus K_h^d \stackrel{\div}{\longrightarrow} W_h{\longrightarrow}0,
\end{equation}
and
\begin{equation}\label{exactZK}
0{\longleftarrow}S_h \stackrel{\div_{h}}{\longleftarrow} Z_h^c\oplus K_h^c \stackrel{\curl_h}{\longleftarrow} Z_h^d\oplus K_h^d \stackrel{\grad_{h}}{\longleftarrow} W_h{\longleftarrow}0.
\end{equation}
The space in the end of the arrow is the range of the operator above and in the beginning is the real domain. The precise characterization of the null space $Z_h$ or $K_h$ can be found by tracing back of the corresponding operators.

\subsection{Discrete Formulations of Vector Laplacian.}

 On the continuous level, the mixed formulation of the vector Laplacian in space $\bs H_{0}(\curl)$ is: Find $\sigma \in H_{0}^{1}, \bs u\in \bs H_{0}(\curl)$ such that
\begin{equation}\label{mixvecLap_continuous}
\left\{
\begin{aligned}
-(\sigma, \tau) + (\bs u, \grad \tau) &= 0 & \text{ for all } \tau \in H_{0}^{1},\\
(\grad \sigma, \bs v) + (\curl\bs u, \curl\bs v) &= (\bs f, \bs v) & \text{ for all } \bs v\in \bs H_{0}(\curl).
\end{aligned}
\right.
\end{equation}
The problem \eqref{mixvecLap_continuous} on the discrete level is: Find $\sigma_h \in S_{h},\bs u_h\in \bs U_{h}$ such that
\begin{equation}\label{mixvecLap-discrete}
\left\{
\begin{aligned}
-(\sigma_h, \tau_h) + (\bs u_h, \grad \tau_h) &= 0 & \text{ for all } \tau_h \in S_{h},\\
(\grad \sigma_h, \bs v_h) + (\curl\bs u_h, \curl \bs v_h) &=(\bs f, \bs v_{h}) & \text{ for all } \bs v_h\in \bs U_{h}.
\end{aligned}
\right.
\end{equation}

Note that the first equation of ~\eqref{mixvecLap-discrete} can be interpreted as $\sigma_h = -\div_{h}\bs u_h$ and in the second equation of ~\eqref{mixvecLap-discrete} the term $(\grad \sigma_h,\bs v_h) = - (\sigma_h, \div_{h}\bs v_h)$. After eliminating $\sigma_h$ from the first equation, we can write the discrete vector Laplacian for edge elements as
\begin{equation}\label{disvecLap}
-\Delta_h^{c}  \bs u_h : = \curl_{h} \curl \bs u_h - \grad \div_{h} \bs u_h,
\end{equation}
which is a discretization of the identity
 \begin{equation*}
- \Delta \bs u = \curl \curl \bs u - \grad \div\bs u.
\end{equation*}

Choosing appropriate bases for the finite element spaces, we can represent the spaces $S_{h}$ and $\bs V_{h}$ by $\mbb R^{\dim S_{h}}$ and $\mbb R^{\dim \bs V_{h}}$ respectively. In the following, we shall use the same notation for the vector representation of a function if no ambiguity arises. Then we have the corresponding operator and matrix formulations as: $\mcal L_{h}^{c} :  S_{h}\times \bs U_{h} \rightarrow S_{h}'\times \bs U_{h}'$
\begin{equation}\label{matvecLap}
\mcal L_{h}^{c}
\begin{pmatrix}
\sigma_h\\
\bs u_h
\end{pmatrix}
:=
\begin{pmatrix}
-M_v & B\\
B^T & C^TM_fC
\end{pmatrix}
\begin{pmatrix}
\sigma_h\\
\bs u_h
\end{pmatrix}
=
\begin{pmatrix}
0\\
\bs f
\end{pmatrix}.
\end{equation}
Here $M_v, M_e$ and $M_f$ are mass matrices of the vertex element, edge element and the face element, respectively, $B^T = M_e G$ corresponds to a scaling of the $\grad$ operator $G$, and $C$ to the $\curl$ operator. We follow the convention of Stokes equations to reserve $B$ for the (negative) divergence operator. Note that to form the corresponding matrices of weak derivative operators, the inverse of mass matrices will be involved. The Schur complement
\begin{equation}\label{Ahc}
A_{h}^{c} = B^TM^{-1}_vB + C^TM_fC
\end{equation}
is the matrix representation of discrete vector Laplacian ~\eqref{disvecLap}. The system \eqref{matvecLap} can be reduced to the Schur complement equation
\begin{equation}\label{schur-complement}
A_{h}^{c}\bs  u_{h} = \bs f.
\end{equation}

Similarly, the mixed formulation of the vector Laplacian in space $\bs H_{0}(\div)$ is: Find $\bs\sigma \in \bs H_{0}(\curl),\bs u\in \bs H_{0}(\div)$ such that
\begin{equation}\label{mixvecLap_continuous_div}
\left\{
\begin{aligned}
-(\bs\sigma,\bs \tau) + (\bs u, \curl\bs \tau) &= 0 & \text{ for all }\bs \tau \in \bs H_{0}(\curl),\\
(\curl\bs \sigma,\bs v) + (\div\bs u, \div\bs v) &= (\bs f,\bs v) & \text{ for all } \bs v\in \bs H_{0}(\div).
\end{aligned}
\right.
\end{equation}
The corresponding discrete mixed formulation is:
 Find $\bs \sigma_{h} \in \bs U_{h},\bs u_{h}\in \bs V_{h}$ such that
\begin{equation}\label{mixvecLap-discrete-div}
\left\{
\begin{aligned}
-(\bs\sigma_{h},\bs \tau_{h}) + (\bs u_{h}, \curl \bs\tau_{h}) &= 0 & \text{ for all }\bs \tau_{h} \in \bs U_{h},\\
(\curl \bs\sigma_{h}, \bs v_{h}) + (\div\bs u_{h}, \div\bs v_{h}) &= (\bs f, \bs v_{h}) & \text{ for all } \bs v_{h}\in \bs V_{h}.
\end{aligned}
\right.
\end{equation}

Eliminating $\bs\sigma_{h}$ from the first equation of \eqref{mixvecLap-discrete-div}, we have the discrete vector Laplacian for face elements as
\begin{equation}\label{disvecLap-div}
-\Delta_h^{d} \bs u_h : = \curl \curl_{h} \bs u_h - \grad_{h} \div \bs u_h,
\end{equation}
and the operator and matrix formulations are: $\mcal L_{h}^{d} :  \bs U_{h}\times \bs V_{h} \rightarrow \bs U_{h}'\times \bs V_{h}'$
\begin{equation}\label{matvecLap-div}
\mcal L_{h}^{d}
\begin{pmatrix}
\bs\sigma_h\\
\bs u_h
\end{pmatrix}
:=
\begin{pmatrix}
-M_e & C^{T}\\
C & B^TM_tB
\end{pmatrix}
\begin{pmatrix}
\bs \sigma_h\\
\bs u_h
\end{pmatrix}
=
\begin{pmatrix}
0\\
\bs f
\end{pmatrix},
\end{equation}
where $M_t$ denotes the mass matrix of the discontinuous element. The Schur complement $A_{h}^{d} = CM_{e}^{-1}C^{T} + B^{T}M_t B$ is the matrix representation of discrete vector Laplacian ~\eqref{disvecLap-div}.
Similarly, the reduced equation of \eqref{matvecLap-div} is
\begin{equation}\label{schur-complement-div}
A_{h}^{d}\bs u_{h}=\bs f.
\end{equation}

We shall consider multigrid methods for solving \eqref{schur-complement} and \eqref{schur-complement-div} and use them to construct efficient preconditioners for the corresponding saddle point systems \eqref{matvecLap} and \eqref{matvecLap-div}, respectively.

\subsection{Discrete Poincar\'{e} Inequality and Inverse~Inequality}
In this subsection, we define the norms associated with the discrete vector Laplacian, and prove discrete Poincar\'{e} and inverse inequalities.

\begin{definition}\label{Ah_definition}
For $ \bs u_{h} \in \bs U_{h}$, define $\|\bs u_{h}\|_{A^{c}_{h}}^{2} =a^{c}_{h}(\bs u_{h},\bs u_{h})$, where the bilinear form  $a^{c}_{h}(\cdot, \cdot)$ is defined as
$$a^{c}_{h}( \bs u_{h},\bs  v_{h}) :=  (\curl\bs u_{h}, \curl\bs v_{h})+(\div_{h}\bs u_{h}, \div_{h}\bs v_{h}).$$
Similarly, for $\bs u_{h} \in \bs V_{h}$, define $\|\bs u_{h}\|_{A^{d}_{h}}^{2} =a^{d}_{h}(\bs u_{h}, \bs u_{h})$, where the bilinear form  $a^{d}_{h}(\cdot, \cdot)$ is defined as
$$a^{d}_{h}(\bs u_{h}, \bs v_{h}) := (\curl_{h} \bs u_{h}, \curl_{h}\bs  v_{h})+(\div\bs u_{h}, \div\bs v_{h}).$$
\end{definition}
\begin{lemma}[Discrete Poincar\'{e} ~Inequality]
\label{discrete poincare_lemma}
We have the following discrete Poincar\'{e} inequalities:
\begin{align}
\label{dP1} \|\bs u_h\|\lesssim \|\bs u_h\|_{A^{c}_h} \quad \hbox{for all} ~\bs u_h \in \bs U_h;\\
\label{dP2} \|\bs u_h\|\lesssim \|\bs u_h\|_{A^{d}_h} \quad \hbox{for all} ~\bs u_h \in \bs V_h.
\end{align}
\end{lemma}
\begin{proof}
We prove the first inequality \eqref{dP1} and refer to \cite{Chen2014a} for a proof of \eqref{dP2}. From the discrete Hodge decomposition, we have for $\bs u_h \in \bs U_h$, there exist $\rho \in S_{h}$ and $\bs \phi \in Z_h^d$ such that
\begin{align}
\bs u_h = \grad \rho + \curl_{h} \bs \phi.
\label{eq:hodge discrete}
\end{align}
Applying $-\div_{h}$ to \eqref{eq:hodge discrete}, we have $-\div_{h}\bs u_h = -\div_{h} \grad \rho$, thus
$$
\|\grad \rho\|^{2} = (-\div_{h}\bs u_h, \rho) \leq \|\div_{h} \bs u_h\| \|\rho\|\lesssim \|\div_{h}\bs u_h\| \|\grad \rho\|,
$$
which leads to
\begin{equation}\label{ieq:1}
\|\grad \rho\| \lesssim \|\div_{h}\bs u_h\|.
\end{equation}


To control the other part, we first prove a discrete Poincar\'e inequality in the form
\begin{equation}\label{phiPoincare}
\|\bs \phi\|\lesssim  \|\curl_{h}\bs \phi\| \quad \text{ for all }\phi \in Z_h^d.
\end{equation}
%
%
By the exactness of the complex \eqref{exactZK}, there exists $ \bs v \in K_h^c$ such that $\bs \phi = \curl\bs v$.
We recall another Poincar\'{e} ~inequality \cite{Monk.P2003,Hiptmair.R2002}
$$
\| \bs v\| \lesssim \|\curl \bs v\| \quad \hbox{for all} ~\bs v\in K_h^c = \bs U_h\cap \text{ker}(\curl)^{\bot}.
$$
Then we have
$$
\|\bs\phi\|^{2} = (\bs\phi, \curl \bs v) = ( \curl_{h}\bs \phi,\bs v) \leq \|\curl_{h}\bs \phi\|\| \bs v\|
\lesssim \|\curl_{h} \bs \phi\|\|\curl\bs v\| =  \|\curl_{h} \bs \phi\|\|\bs\phi\|.
$$
Canceling one $\|\bs \phi\|$, we obtain the desired inequality \eqref{phiPoincare}.

Applying $\curl$ to the Hodge decomposition \eqref{eq:hodge discrete} and using the inequality \eqref{phiPoincare}, we have $\curl\bs u_h = \curl \curl_{h} \bs\phi$, thus
$$
\|\curl_{h} \bs \phi\|^{2} = (\curl\bs u_h, \bs \phi)
\leq \|\curl\bs u_h\|\|  \bs\phi\| \lesssim \|\curl\bs u_h\|\| \curl_{h} \bs \phi\|,
$$
which leads to the inequality
\begin{equation}\label{ieq:2}
\|\curl_{h} \bs \phi\| \lesssim \|\curl \bs u_h\|.
\end{equation}

Combine inequalities \eqref{ieq:1} and \eqref{ieq:2}, we have proved that
$$
\|\bs u_h\| \leq \| \grad \rho \| + \|\curl_{h} \bs \phi\|\lesssim \|\div_{h}\bs u_h\| +  \|\curl\bs u_h\| \lesssim \|\bs u_h\|_{A^{c}_{h}}.
$$
%
\end{proof}
\begin{remark}\rm
 The result and the proof can be easily generalized to mixed discretization of Hodge Laplacian in discrete differential forms \cite{Arnold2006}. We keep the concrete form in $\bs H(\curl)$ and $\bs H(\div)$ conforming finite element spaces for the easy access of these results. \qed
\end{remark}

It is easy to prove the following inverse inequalities:
\begin{align*}
\|\bs u_h\|_{A^{c}_h}\lesssim h^{-1}\|\bs u_h\| \quad \hbox{for all} ~\bs u_h \in \bs U_h;\\
\|\bs u_h\|_{A^{d}_h} \lesssim h^{-1}\|\bs u_h\|\quad \hbox{for all} ~\bs u_h \in \bs V_h.
\end{align*}
\section{Multigrid Methods for Discrete Vector Laplacian}
In this section, we describe a variable V-cycle multigrid algorithm to solve the Schur complement equations \eqref{schur-complement} and \eqref{schur-complement-div}, and prove that it is a good preconditioner.

\subsection{Problem Setting}
Let us assume that nested tetrahedral partitions of $\Omega$ are given as
$$
\mathcal T_1 \subset \cdots \subset \mathcal T_J = \mathcal T_h,
$$
and the corresponding $H_0^1$, $\bs H_0(\curl)$ and $\bs H_0(\div)$ finite element spaces are
\begin{align*}
S_1\subset \cdots \subset S_J = S_{h}, \quad \bs U_1 \subset \cdots \subset\bs U_J = \bs U_{h}, \quad \bs V_1\subset\cdots\subset\bs V_J = \bs V_{h}.
\end{align*}
For a technical reason, we assume that the edge element space and the face element space contain the full linear polynomial which rules out only the lowest order case. When no ambiguity can arise, we replace subscripts $h$ by the level index $k$ for $k=1,2,\ldots, J$.

The discretization \eqref{mixvecLap_continuous} of the mixed formulation of the vector Laplacian in space $\bs H_{0}(\curl)$ based on $\mathcal T_k$, for $k=1,2,\ldots, J$, can be written as
\begin{equation}\label{eq:dismxivec_on_k}
\begin{pmatrix} -M_{v,k}  &  B_k  \\ B_k^T  & C_k^TM_{f,k}C_k
\end{pmatrix}
\begin{pmatrix}
\sigma_k \\ \bs u_k  \end{pmatrix}
= \begin{pmatrix}  0 \\\bs f_k
\end{pmatrix}.
\end{equation}
Eliminating $\sigma_k$ from \eqref{eq:dismxivec_on_k}, we  get the reduced Schur complement equation
\begin{equation}\label{eq:dis_eliminating_sigma_k}
A_k^{c}\bs u_k = (B_k^TM_{v,k}^{-1}B_k + C_k^T M_{f,k} C_k ) \bs u_k = \bs f_k.
\end{equation}
The discretization \eqref{mixvecLap_continuous_div} of the mixed formulation of vector Laplacian in space $\bs H_{0}(\div)$ on $\mathcal T_k$, for $k=1,2,\ldots, J$, can be written as
\begin{equation}\label{eq:dismxivec_on_k_div}
\begin{pmatrix}
-M_{e,k}  &  C_k^{T}  \\
C_k  & B_k^{T}M_{t,k}B_k
\end{pmatrix}
 \begin{pmatrix}\bs \sigma_k \\ \bs u_k  \end{pmatrix}
= \begin{pmatrix}  0 \\ \bs f_k  \end{pmatrix},
\end{equation}
and the reduced Schur complement equation is
\begin{equation}\label{eq:dis_eliminating_sigma_k_div}
A_k^{d}\bs  u_k = (B_k^TM_{t,k}B_k + C_k M_{e,k}^{-1} C_k^{T} )\bs u_k = \bs f_k.
\end{equation}
We are interested in preconditioning the Schur complement equations \eqref{eq:dis_eliminating_sigma_k} and \eqref{eq:dis_eliminating_sigma_k_div} in the finest level, i.e., $k=J$.

Notice that, for $k<J$, $A_{k}^{c}$ and $A_{k}^{d}$ are defined by the discretization of the vector Laplacian on the trianglulation $\mcal T_{k}$, but not by the Galerkin projection of $A_{J}^{c}$ or $A_{J}^{d}$ since the inverse of a mass matrix  is involved. In other words, $A_{k}^{c}$ and $A_{k}^{d}$ are non-inherited from  $A_{J}^{c}$ or $A_{J}^{d}$ for $k<J$.

When necessary, the notation without the superscript $c$ and $d$ is used to unify the discussion. The notation $\mcal V_{k}$ is used to represent both $\bs U_{k}$ and $\bs V_{k}$  spaces.

\subsection{A Variable V-cycle Multigrid Method}
We introduce some operators first. Let $R_{k}$ denote a smoothing operator on level $k$, which is assumed to be symmetric and convergent. Let $I^{k}$ denote the prolongation operator from level $k-1$ to level $k$, which is the natural inclusion since finite element spaces are nested. The transpose $Q_{k-1} = (I^k)^T$ then represents the restriction from level $k$ to level $k-1$. The Galerkin projection $P_{k-1}$, which is from level $k$ to level $k-1$, is defined as: for any given $\bs u_{k} \in \mcal V_{k}, P_{k-1}\bs u_k\in \mcal V_{k-1}$ satisfies
$$
a_{k-1}(P_{k-1} \bs u_k,  \bs v_{k-1}) =  a_k( \bs u_k, I^k \bs v_{k-1}) =  a_{k}( \bs u_k,\bs v_{k-1})\quad \text{for all}~\bs v_{k-1} \in \mcal V_{k-1}.
$$

The variable V-cycle multigrid algorithm is as following.\\
\underline{\hspace{12.6cm}}
\newline
\textbf{Algorithm 2.} Multigrid Algorithm: $\bs u_{k}^{MG} = MG_{k}(\bs f_{k}; \bs u_{k}^{0}, m_{k})$\\
\underline{\hspace{12.6cm}}
\newline
Set $MG_{1} =  A_1^{-1}$. \\
For $k\geq 2$, assume that $MG_{k-1}$ has been defined. Define $MG_{k}(\bs f_{k}; \bs u_{k}^{0}, m_{k})$ as follows:
\newline
\begin{itemize}
  \item Pre-smoothing: Define $\bs u_{k}^{l}$ for $l=1, 2, \cdots, m_{k}$ by
  $$ \bs u_k^l =  \bs u_k^{l-1} + R_k( \bs f_k - A_k \bs u_k^{l-1}).$$
  \item Coarse-grid correction: Define $\bs u_k^{m_{k}+1} = \bs u_k^{m_{k}} + I^{k}\bs e_{k-1}$, where
  $$ \bs e_{k-1} = MG_{k-1}(Q_{k-1}(\bs f_k - A_k \bs u_k^{m_{k}}); 0, m_{k-1}).$$
 \item Post-smoothing: Define $\bs u_{k}^{l}$ for $l=m_{k}+ 2, \cdots, 2m_{k}+1$ by
  $$ \bs u_k^l =  \bs u_k^{l-1} + R_k(\bs  f_k - A_k \bs u_k^{l-1}).$$
  Define $\bs u_{k}^{MG} = \bs u_{k}^{2m_{k}+1}$.
\end{itemize}
\underline{\hspace{12.6cm}}

In this algorithm, $m_{k}$ is a positive integer which may vary from level to level, and determines the number of smoothing iterations on the $k$-th level, see \cite{Bramble.J;Pasciak.J;Xu.J1991,Bramble.J;Pasciak.J1992}.

\subsection{Multigrid Analysis Framework}
We employ the multigrid analysis framework developed in~\cite{Bramble.J;Pasciak.J;Xu.J1991}. Denoted by $\lambda_{k}$ the largest eigenvalue of $A_{k}$. For the multigrid algorithm to be a good preconditioner to $A_k$, we need to verify the following assumptions:
\begin{description}
  \item[(A.1)] ``Regularity and approximation assumption'': For some $0<\alpha \leq 1$,
  $$
  \left|a_k((I-P_{k-1})\bs u_k, \bs u_k)\right| \leq C_A\left(\frac{\|A_k\bs u_k\|^2}{\lambda_k}\right)^{\alpha}a_k( \bs u_k, \bs u_k)^{1-\alpha}\qquad\text{for all }\bs u_k\in\mcal V_k,
  $$
  holds with  constant $C_{A}$ independent of $k$;
  \item[(A.2)] ``Smoothing property'':
  $$
  \frac{\|\bs u_k\|^2}{\lambda_k}\leq C_R (R_k \bs u_k, \bs u_k)\qquad \text{for all } \bs u_k\in \mcal V_k,
  $$
 holds with  constant $C_{R}$ independent of $k$.
\end{description}

Following the standard arguments, we can show that the largest eigenvalue of $A_k$, $\lambda_k$, satisfies $\lambda_k \eqsim h_k^{-2}$ for $k=1,2,\ldots, J$.


\subsection{Smoothing Property}
The symmetric Gauss-Seidel (SGS) or a properly weighted Jacobi iteration both satisfy the smoothing property (A.2), a proof of which can be found in \cite{Bramble.J;Pasciak.J1992}. For completeness we present a short proof below.

Recall that Gauss-Seidel iteration can be understood as a successive subspace correction method applied to the basis decomposition $\mcal V_k = \sum _{i=1}^{N_k}\mcal V_{k,i}$ with exact local solvers \cite{Xu1992}. For $\bs u\in \mcal V_k$, let $\bs u = \sum_{i=1}^{N_k}\bs u_i$ be the basis decomposition. By the X-Z identity \cite{Xu.J;Zikatanov.L2002,Chen.L2009c} for the multiplicative method, we have
$$
(R^{-1}_{\rm \small SGS} \bs u,\bs  u) = \|\bs u\|_{A_k}^2 + \sum_{i=0}^N\|P_i\sum _{j>i}\bs u_j\|_{A_k}^2,
$$
where $P_i$ is the $A_k$ orthogonal projection to $\mcal V_{k,i}$.
For an index $i$, we denote by $n(i)$ the set of indices such that the corresponding basis function has overlapping support with basis function at $i$. We then estimate the second term as
$$
\sum_{i=0}^N\|P_i\sum _{j>i}\bs u_j\|_{A_k}^2\leq \sum_{i=0}^N\sum _{j\in n(i)}\|\bs u_j\|_{A_k}^2\lesssim  \lambda_k\sum_{i=0}^N\|\bs u_i\|^2\lesssim \lambda_k\|\bs u\|^2.
$$
Here we use the sparsity of $A_k$ such that the repetition in the summation, i.e, the number of indices in $n(i)$, is uniformly bounded above by a constant. The last step is from the stability of the basis decomposition in $L^2$-norm which holds for all finite element spaces under consideration.

We have thus proved that $(R^{-1}_{\rm \small SGS}\bs u, \bs u) \lesssim \lambda_k\|\bs u\|^2$ which is equivalent to the smoothing property by a simple change of variable. Similar proof can be adapted to the weighted Jacobi smoother.

\subsection{Regularity Results}
In this subsection, we will present some regularity results for Maxwell equation. Recall that, we assume $\Omega$ is a bounded and convex polyhedron throughout of this paper.

\begin{lemma}[Theorem 3.7 and 3.9 in \cite{PGirault1986}]\label{lm:GR}
 The space $\bs H(\div; \Omega)\cap \bs H_0(\curl; \Omega)$ and  $\bs H_0(\div; \Omega)\cap \bs H(\curl; \Omega)$ are continuously imbedded into $\bs H^1(\Omega)$ and 
 $$
 \|\bs \phi\|_1 \lesssim \|\curl \bs \phi\| + \|\div \bs \phi\|.
 $$
 for all functions $\bs \phi \in \bs H(\div; \Omega)\cap \bs H_0(\curl; \Omega)$ or $\bs H_0(\div; \Omega)\cap \bs H(\curl; \Omega)$.
\end{lemma}

In the sequel, we are going to develop an $H^2$ regularity result of Maxwell equation.
\begin{lemma}\label{lem:regularity_pro}
For functions $\bs \psi \in \bs H(\div; \Omega)\cap \bs H_0(\curl; \Omega)$ or $\bs H_0(\div; \Omega)\cap \bs H(\curl; \Omega)$. satisfying $\curl\bs\psi\in \bs H^1(\Omega)$ and $\div\bs\psi\in \bs H^1(\Omega)$. Then $\bs \psi \in H^2(\Omega)$ and
$$
\|\bs\psi\|_2 \lesssim \|\curl\bs\psi\|_1 + \|\div\bs\psi\|_1.
$$
\end{lemma}
\begin{proof}
Let $\tilde{\bs\psi}$ be the zero extension of $\bs\psi$ from $\Omega$ to $\mathbb R^3$ and $\mathcal F\tilde {\bs \psi}$ denote the Fourier transform of $\tilde \psi$ defined as usual by
$$
\mathcal F\tilde {\bs \psi} = \int_{\mathbb R^3}e^{-2i\pi(x,\mu)}\tilde {\bs\psi}\text{d} x,\qquad (x,\mu) = \sum\limits_{i =1}^3x_i\mu_i.
$$
By carefully calculation, we can prove that
$$
\left \|\mathcal F\frac{\partial^2\tilde{\psi_l}}{\partial x_i\partial x_j} \right \| \lesssim \|\curl\bs\psi\|_1 + \|\div\bs\psi\|_1.
$$
The desired result follows by the properties of Fourier transform.
\end{proof}

Then we have the following $H^2$ regularity of Maxwell equation.
\begin{lemma}\label{lem:regularity}
For any $\bs\psi\in K^c$, define $\bs\zeta\in K^c$ to be the solution of 
\begin{equation}\label{curlzeta}
(\curl\bs\zeta,\curl\bs\theta) = (\bs\psi,\bs\theta)\quad \text{for all }\bs\theta\in K^c.
\end{equation}
Then $\curl\bs\zeta \in \bs H^2(\Omega)$ and 
\begin{align}
\label{H1reg} \|\curl\bs\zeta\|_1 &\lesssim \|\bs\psi\|,\\
\label{H2reg} \|\curl\bs\zeta\|_2 &\lesssim \|\curl\bs\psi\|.
\end{align}
\end{lemma}
\begin{proof}
Indeed $\curl\bs\zeta \in \bs H_0(\div; \Omega)$ with $\div \curl\bs\zeta = 0$ and \eqref{curlzeta} implies $\curl^w\curl\bs\zeta = \bs \psi$ holds in $L^2$. The desired $H^1$ regularity \eqref{H1reg} of $\curl\bs\zeta$ then follows from Lemma \ref{lm:GR}.

For any $\bs w\in \bs H_0(\div)$, let $\theta = \curl^w\bs w$. Then  equation \eqref{curlzeta} implies
$$
(\curl\curl^w\curl\bs \zeta,\bs w) = (\curl\bs \psi,\bs w)\qquad\text{for all }\bs w\in \bs H_0(\div).
$$
Thus, we have
$$
\curl\curl^w\curl\bs \zeta = \curl\bs \psi \; \text{in }L^2, \text{ and }\div^w\curl^w\curl\bs \zeta = 0.
$$
Again by Lemma \ref{lm:GR}, it holds
$$
\|\curl^w\curl\bs \zeta\|_1 \lesssim \|\curl\bs\psi\|.
$$
The desired result \eqref{H2reg} is then obtained by Lemma \ref{lem:regularity_pro}.
\end{proof}

\subsection{Error Estimate of Several Projection Operators}
We define several projection operators to the null space $K_h^{\mcal D}$. Given $u\in H(\mcal D)$, find $P_h^{\mcal D} u \in K_h^{\mcal D}$ such that
\begin{equation}\label{PhD}
(\mcal D P_h^{\mcal D} u, \mcal D v_h) = (\mcal D u, \mcal D v_h), \quad \text{for all }v_h \in K_h^{\mcal D}.
\end{equation}
Equation \eqref{PhD} determines $P_h^{\mcal D} u$ uniquely since $(\mcal D \cdot , \mcal D \cdot )$ is an inner product on the subspace $K_h^{\mcal D}$ which can be proved using the Poincar\'e inequality (Lemma \ref{discrete poincare_lemma}). For $\mcal D = \grad$, we understand $K_h^{\grad}$ as $S_h$. 

\begin{lemma}[Theorem 2.4 in Monk \cite{Monk.P1992}]\label{lem:app_maxwell}
Suppose that $\curl \bs u \in \bs H^{k}$ and let $\bs u_h = P_h^c \bs u$ to $\bs U_h$ which contains polynomial of degree less than or equal to $k$. Then we have the error estimate
 $$
 \|\curl(\bs u - \bs u_h)\| \lesssim h^r\|\curl \bs u\|_{r},\qquad \text{for } 1 \leq r \leq k.
 $$
\end{lemma}

We are also interested in the estimate of projections between two consecutive finite element spaces. Following the convention of multigrid community, for any $2< k\leq J$, let $\mathcal T_H = \mathcal T_{k-1}$ and $\mathcal T_h = \mathcal T_k$. Notice that the ratio $H/h \leq C$. 



 The following error estimates are obtained in~\cite{Arnold.D;Falk.R;Winther.R2000}.
  \begin{lemma}\label{lem:L2_Z_h}
 Given $\bs u_h\in K_h^c$, let $\bs u_H = P_H^c\bs u_h$. Then
 \begin{align*}
 \|\bs u_h - \bs u_H\| & \lesssim  H\|\curl \bs u_h\|,\\
 \|\curl (\bs u_h - \bs u_H)\| & \lesssim  H\|\curl_h\curl\bs u_h\|.
 \end{align*}
 \end{lemma}

\begin{lemma}\label{lem:app_grad_w_h}
 Give $\bs v_h\in K_h^d$, let $\bs v_H = P_H^d\bs v_h$. Then
 \begin{align*}
 \|\bs v_h - \bs v_H\| & \lesssim  H\|\div\bs v_h\|,\\
 \|\div(\bs v_h - \bs v_H)\| & \lesssim  H\|\grad_h\div\bs v_h\|.
 \end{align*}
 \end{lemma}

We now introduce a projection to $K^c$. Let 
$Q_{K}^c: \bs L^2 \to K^c$ be the $L^2$-projection to $K^c$. Notice that for $\bs u\in \bs L^2$, $Q_K^c \bs u = \bs u - \nabla p$ where $p\in H_0^1$ is determined by the Poisson equation $(\nabla p, \nabla q) = (\bs u, \nabla q)$ for all $q\in H_0^1$. Therefore $\curl Q_K^c \bs u = \curl \bs u$. Similarly we define $Q_h^c: \bs L^2 \to K_h^c$ as $Q_h^c \bs u = \bs u - \nabla p$ where $p\in S_h$ is determined by the Poisson equation $(\nabla p, \nabla q) = (\bs u, \nabla q)$ for all $q\in S_h$. We have the error estimate, c.f. \cite{Arnold.D;Falk.R;Winther.R2000,Zhou2013}. 
\begin{lemma}
For $\bs u_h \in K_h^c$, we have
\begin{equation}\label{Puh}
\| Q_K^c \bs u_h - \bs u_h \|\lesssim h\|\curl \bs u_h\|.
\end{equation} 
And for $\bs u_H \in K_H^c$
\begin{equation}\label{PuH}
\| Q_h^c \bs u_H - \bs u_H \|\lesssim H\|\curl \bs u_H\|.
\end{equation}
\end{lemma}
In the estimate \eqref{Puh}-\eqref{PuH}, we lift a function in a coarse space to a fine space while in Lemma \ref{lem:L2_Z_h}, we estimate the projection. The $L^2$-projection $Q_h^c: K_H^c \to K_h^c$ can be thought of as a prolongation of non-nested spaces $K_H^c$ and $K_h^c$.

\subsection{Approximation Property of Edge Element Spaces}
Let $\bs u_h\in \bs U_{h}$ be the solution of equation
\begin{equation}\label{eq:solutionof_h}
a_h^{c}(\bs u_h,\bs v_h) = (\bs f_h, \bs v_h)\qquad\text{for all }\bs v_h\in\bs U_{h},
\end{equation}
and $\bs u_H\in \bs U_{H}\subset \bs U_{h}$ be the solution of equation
\begin{equation}\label{eq:solutionof_H}
a_H^{c}(\bs u_H,\bs v_H) = (\bs f_h,\bs v_H)\qquad\text{for all }\bs v_H\in\bs U_{H}.
\end{equation}
We have the Hodge decomposition
\begin{align}
\label{eq:u_h_decomp}
 \bs u_h & = \grad \phi_h \oplus^{\bot} \bs u_{0,h}, \quad \text{ with unique }\phi_h\in S_{h},\  \bs u_{0,h}\in K_{h}^c,\\
 \label{eq:u_H_decomp}
 \bs u_H & = \grad \phi_H \oplus^{\bot} (\bs u_{0,H} + \bs e_H), \quad \text{ with unique }\phi_H\in S_{H},\  \bs u_{0,H}\text{~and~}\bs e_H\in K_H^c,\\
 \label{eq:f_h_decomp}
 \bs f_h & =  \grad g_h \oplus^{\bot} \curl_h \bs q_h , \quad \text{ with unique } g_h\in S_{h}, \ \bs q_h\in Z_h^d,
\end{align}
where $\bs u_{0,H} = P_H^c \bs u_{0,h}$ . 
Then by Lemma \ref{lem:L2_Z_h}, we immediately get the following estimate.
\begin{lemma}\label{lem:u_0_H_L2}
Let $\bs u_{0,h} $ and $\bs u_{0,H}$ be defined as in equations \eqref{eq:u_h_decomp} and \eqref{eq:u_H_decomp}. It holds
$$ \|\bs u_{0,h} - \bs u_{0,H}\| \lesssim H\|\bs u_h\|_{A_h^c}.  $$
\end{lemma}
Now we turn to the estimate of $\bs e_H$ being given in equation \eqref{eq:u_H_decomp}.
\begin{lemma}\label{lem:e_H_app}
  Let $\bs e_H\in K_H^c$ be defined as in equation \eqref{eq:u_H_decomp}. It holds
  $$  \|\bs e_H\|_{A_h^c} \lesssim H \|A_h^c \bs u_h\|.$$
\end{lemma}

\begin{proof}
By equations \eqref{eq:solutionof_h} and \eqref{eq:solutionof_H}, we have
\begin{align*}
(\curl \bs u_{0,h}, \curl \bs v_{h}) & =  (\bs q_h,\curl\bs v_{h}), \quad \text{for all }\bs v_{h}\in K_h^c\\
(\curl(\bs u_{0,H} + \bs e_H),\curl\bs v_H) & = (\grad g_h, \bs v_H) + (\bs q_h, \curl\bs v_H) , \quad \text{for all }\bs v_{H}\in K_H^c,
\end{align*}
where $g_h$ and $\bs q_h$ are defined in equation \eqref{eq:f_h_decomp}. Then 
\begin{equation}
  \label{eq:e_H_eq}
  (\curl\bs e_H,\curl\bs v_H) = (\grad g_h,\bs v_H)\qquad \text{for all }\bs v_H \in K_H^c.
\end{equation}
Let $\bs e_h = Q_h^c\bs e_H$, then $\div_h\bs e_h = 0$ and by Lemma \ref{lem:L2_Z_h} $\|\bs e_h - \bs e_H\| \lesssim H\|\curl\bs e_H\|.$
Thus it holds
\begin{eqnarray*}
  (\curl\bs e_H,\curl\bs e_H) & = & (\grad g_h,\bs e_H) = (\grad g_h,\bs e_H - \bs e_h) \lesssim H \|\grad g_h\|\|\curl\bs e_H\|,
\end{eqnarray*}
which implies
$$
\|\curl\bs e_H\| \lesssim H\|\grad g_h\| \leq H\|\bs f_h\| = H\|A_h^c\bs u_h\|.
$$

Using the fact that $\div_h\bs e_h = 0$, the inverse inequality and the above inequality, we immediately get
$$
\|\div_h\bs e_H\| = \|\div_h(\bs e_H - \bs e_h)\| \lesssim h^{-1} \|\bs e_h - \bs e_H\| \lesssim \frac{H}{h} \|\curl\bs e_H\| \lesssim H \|A_h^c\bs u_h\|.
$$
The desired result then follows.
\end{proof}

We now explore the relation between $\phi_h, \phi_H$, and $g_h$ defined in equations \eqref{eq:u_h_decomp}-\eqref{eq:f_h_decomp}.
\begin{lemma}
  \label{lem:phi_app}
  Let $\phi_h\in S_h$ and $\phi_H\in S_H$ be defined as in equations \eqref{eq:u_h_decomp} and \eqref{eq:u_H_decomp}. It holds
  $$
  \|\grad \phi_h - \grad\phi_H\| \lesssim H \|\bs u_h\|_{A_h^c}.
  $$
\end{lemma}

\begin{proof}
For equation \eqref{eq:solutionof_h}, test with $\bs v_{h} \in \grad S_{h}$ to get
\begin{align*}
(\div_{h} \grad \phi_{h}, \div_{h}\bs v_{h}) = (\grad g_{h},\bs v_{h}) = -(g_{h}, \div_{h}\bs v_{h}),
\end{align*}
which implies $-\div_{h} \grad \phi_{h}=-\div_{h} \bs u_{h} = g_{h}$, i.e.,
\begin{eqnarray}\label{eq:solution_of_h_}
-\Delta _{h} \phi_{h} = g_{h}.
\end{eqnarray}
From equation \eqref{eq:solution_of_h_}, we can see that $\phi_h$ is the Galerkin projection of $\phi$ to $S_h$, where $\phi\in H_0^1(\Omega)$ satisfies the Poisson equation:
$$
 -\Delta \phi = g_h.
$$
Therefore by the standard error estimate of finite element methods, we have
$$
\|\nabla \phi - \nabla \phi_h\|\lesssim H\|g_h\|.
$$


For equation \eqref{eq:solutionof_H}, choose $\bs v_{H} = \grad \psi_H \in \grad S_{H}$, we have
\begin{align*}
(\div_{H} \grad \phi_{H}, \div_{H}\bs  v_{H}) = (\grad g_{h}, \grad \psi_H) = (\grad P_{H}^{g}g_{h}, \grad \psi_H) ,
\end{align*}
which implies $-\div_{H} \grad \phi_{H} = P_{H}^{g}g_{h}$, i.e.,
\begin{eqnarray}\label{eq:solution_of_H_}
-\Delta _{H} \phi_{H} = P_{H}^{g}g_{h}.
\end{eqnarray}
From equation \eqref{eq:solution_of_H_}, we can see that $\phi_H$ is the Galerkin projection of $\tilde \phi$ to $S_H$, where $\tilde \phi\in H_0^1(\Omega)$ satisfies the Poisson equation:
$$
 -\Delta  \tilde \phi = P_{H}^{g}g_h.
$$
The $H^1$-projection $P_H^g$ is not stable in $L^2$-norm. Applied to functions in $S_h$, however, we can recover one as follows
$$
\|(I- P_H^g)g_h\| \lesssim H\|\grad (I- P_H^g)g_h\|\lesssim H\|\grad g_h\|\lesssim H/h \|g_h\| \lesssim \|g_h\|.
$$
In the last step, we use the fact that the ratio of the mesh size between consecutive levels is bounded, i.e., $H/h\leq C$.

We then have
$$
\|\grad (\tilde \phi - \phi_H)\|\lesssim H\|P_{H}^{g}g_h\| \leq H\|g_h\| + H\|(I- P_H^g)g_h\|\lesssim H\|g_h\|.
$$
And by the triangle inequality and the stability of the projection operator $P_H^{g}$
\begin{align*}
\|\grad(\phi_h - \phi_{H})\| &\leq  \|\grad(\phi_h - \phi)\|+\|\grad(\phi_H - \tilde\phi)\| + \|\grad(\phi-\tilde \phi)\| \\
&\lesssim H\|g_h\|+\|g_{h}-P_{H}^{g}g_{h}\|_{-1}.
\end{align*}
Using the error estimate of negative norms and the inverse inequality, we have
\begin{align*}
\|g_{h}-P_{H}^{g}g_{h}\|_{-1}&\lesssim  H^{2}\|g_{h}\|_{1} \lesssim H\|g_{h}\|.
\end{align*}
Here we use $H^{-1}$ norm estimate for $S_{H}$ having degree greater than or equal to $2$. Noticing that $g_h = \div_h \bs u_h$, we thus get
\begin{eqnarray}\label{eq:H_1_appro_phi}
\|\grad(\phi_{h}-\phi_{H})\|\lesssim H\|\div_{h}\bs u_{h}\| \lesssim H\|\bs u_h\|_{A_h^c}.
\end{eqnarray}
\end{proof}

As a summary of the above results, we have the following approximation result.
\begin{theorem}
  \label{them:App_curl}
  Condition (A.1) holds with $\alpha = \frac{1}{2}$, i.e. for any $\bs u_k\in\bs U_k$, there hold
\begin{equation}\label{eq:A_1_curl} 
a_k^c((I-P_{k-1}) \bs u_k,\bs  u_k) \lesssim  \left(\frac{\|A_k^c\bs u_k\|^2}{\lambda_k}\right)^{\frac{1}{2}} a_k^c( \bs u_k,\bs u_k)^{\frac{1}{2}}.
\end{equation}
\end{theorem}
\begin{proof}
  We use $h$ to denote $k$ and $H$ to denote $k-1$. Let 
  $\bs u_h$, $\bs u_H$, and $\bs f_h$ as in equations \eqref{eq:solutionof_h}-\eqref{eq:solutionof_H} which have Hodge decompositions, c.f. \eqref{eq:u_h_decomp}-\eqref{eq:f_h_decomp}. Let $\delta_1 =\bs u_{0,h} - \bs u_{0,H}$, $\delta_2 = \grad \phi_h - \grad\phi_H$, by Lemmas \ref{lem:u_0_H_L2}, \ref{lem:e_H_app} and \ref{lem:phi_app}, it holds
  \begin{align*}
    a_h^c((I - P_H)\bs u_h, \bs u_h) & = a_h^c(\delta_1,\bs u_h) + a_h^c(\delta_2,\bs u_h) + a_h^c(\bs e_H,\bs u_h) \\
    & \leq \|\delta_1\| \|A_h^c \bs u_h\| + \|\delta_2\| \|A_h^c \bs u_h\| + \|\bs e_H\|_{A_h^c}\|\bs u_h\|_{A_h^c} \\
    & \lesssim H\|\bs u_h\|_{A_h^c} \|A_h^c\bs u_h\|.
  \end{align*}
\end{proof}

\subsection{Approximation Property of Face Element Spaces}
Let $\bs u_h\in \bs V_{h}$ be the solution of equation
\begin{equation}\label{eq:solutionof_h_div}
a_h^{d}(\bs u_h,\bs v_h) = (\bs f_h,\bs v_h)\qquad\text{for all } \bs v_h\in\bs V_{h},
\end{equation}
and $\bs u_H\in \bs V_{H}\subset \bs V_{h}$ be the solution of equation
\begin{equation}\label{eq:solutionof_H_div}
a_H^{d}(\bs u_H,\bs v_H) = (\bs f_h,\bs v_H)\qquad\text{for all } \bs v_H\in\bs V_{H}.
\end{equation}
We can easily see that $\bs f_h = A_h^d\bs u_h$. 

By the Hodge decomposition, we have
\begin{align}
  \label{eq:u_h_decomp_div}
  \bs u_h & = \curl \bs\phi_h \oplus \bs u_{0,h}, \quad  \text{ with unique }\bs\phi_h\in K_h^c,\ \bs u_{0,h}\in K_h^d,\\
  \label{eq:u_H_decomp_div}
\bs u_H & = \curl \bs\phi_H \oplus (\bs u_{0,H} + \bs e_H) \quad  \text{ with unique } \bs\phi_H\in K_H^c,\  \bs u_{0,H},\bs e_H\in K_H^d, \\
\label{eq:f_h_decomp_div}
\bs f_h & = \curl \bs g_h \oplus \grad_h  q_h  \quad \text{~with unique~} \bs g_h\in K_h^c, \ q_h\in W_{h},
\end{align}
where $\bs u_{0,H} = P_H^d \bs u_{0,h}$. 
By Lemma \ref{lem:app_grad_w_h}, we immediately have the following result.
\begin{lemma}
  \label{lem:u_0_H_L2_div}
  Let $\bs u_{0,h}\in\grad_h W_h$ and $\bs u_{0,H}\in \grad_H W_H$ be defined as in equations \eqref{eq:u_h_decomp_div} and \eqref{eq:u_H_decomp_div}. It holds
  $$ \|\bs u_{0,h} - \bs u_{0,H}\| \lesssim H \|\div\bs u_{0,h}\|.  $$
\end{lemma}

The estimate of $\bs e_H\in K_H^d$ defined in equation \eqref{eq:u_H_decomp_div} can be proved analog to Lemma \ref{lem:e_H_app} and thus skipped. 
\begin{lemma}\label{lem:e_H_div}
Assume that $\bs e_H \in \grad_H W_H$ be defined as in equation \eqref{eq:u_H_decomp_div}. Then it holds
$$  \| \bs e_H\|_{A_h^d} \lesssim \|A_h^d \bs u_h\|. $$
\end{lemma}

We now explore the relation between $\bs \phi_h,\ \bs\phi_H$, and $\bs g_h$ defined in equations \eqref{eq:u_h_decomp_div}-\eqref{eq:f_h_decomp_div}.

\begin{lemma}
  \label{lem:maxwell_app}
  Assume that $\bs \psi_h\in K_h^c$. Let $\bs\zeta_h\in K_h^c$ be the solution of equation
  $$
  (\curl\bs\zeta_h,\curl\bs\tau_h) = (\bs \psi_h,\bs\tau_h)\qquad\text{for all }\bs\tau_h \in K_h^c,
  $$ 
  and let $\bs\zeta\in  K^c$ be the solution of equation 
  $$
  (\curl\bs\zeta,\curl\bs\tau)    =     (Q_{ K}^c\bs \psi_h ,\bs\tau) \qquad  \text{for all }\bs\tau\in  K^c .
  $$
  Then, it holds
  $$
  \|\curl(\bs\zeta - \bs\zeta_h)\| \lesssim h \|\bs \psi_h\|.
  $$
\end{lemma}
\begin{proof}
Let $\tilde{\bs \zeta} _h = P_h^c \bs \zeta$. 
By Lemma \ref{lem:app_maxwell}, we have
$$
\|\curl(\bs\zeta - \tilde{\bs \zeta}_h)\| \lesssim h\|\curl\bs\zeta\|_1 \lesssim h\|Q_{ K}^c\bs\psi_h\| \lesssim h\|\bs\psi_h\|.
$$
But $\bs\zeta_h\neq \tilde{\bs \zeta}_h$. Indeed by the definition of $\bs\zeta_h$ and $\tilde{\bs \zeta}_h$, we have
$$
(\curl(\bs\zeta_h - \tilde{\bs\zeta}_h),\curl\bs\tau_h) = (\bs\psi_h - Q_{ K}^c\bs\psi_h,\bs\tau_h)\qquad\text{for all }\bs\tau_h \in  K_h.
$$
Thus, with $\delta_h = \bs\zeta_h - \tilde{\bs\zeta}_h$, we have
\begin{align*}
\|\curl(\bs \zeta_h - \tilde{\bs \zeta}_h)\|^2 & =  (\curl(\bs \zeta_h - \tilde{\bs \zeta}_h), \curl \delta_h) =  (Q_{ K}^c\bs \psi_h - \bs\psi_h,\delta_h)\\
& = (Q_{ K}^c\bs \psi_h - \bs\psi_h, \delta_h - Q_{ K}^c\delta_h) \lesssim h\|\bs\psi_h\|\|\curl \delta_h\|.
\end{align*}

%
 The desired result follows by canceling one $\|\curl \delta_h\|$ and the triangle inequality.
\end{proof}

We are in the position to estimate $\bs \phi_h$ and $\bs \phi_H$. 
\begin{lemma}
  \label{lem:phi_app_div}
  Let $\bs \phi_h\in \bs U_h$ and $\bs \phi_H\in \bs U_H$ be defined as in equations \eqref{eq:u_h_decomp_div} and \eqref{eq:u_H_decomp_div}. It holds
  $$
  \|\curl \bs\phi_h - \curl\bs\phi_H\| \lesssim H \|\bs u_h\|_{A_h^d}.
  $$
\end{lemma}
\begin{proof}
  Chose the test function $\bs v_h = \curl \bs w_h$ with $\bs w_h \in \bs U_h$ in equation \eqref{eq:solutionof_h_div} to simplify the left hand side of \eqref{eq:solutionof_h_div} as
$$
  (\curl_h\bs u_h,\curl_h\curl\bs  w_h) = (\curl_h\curl\phi_h,\curl_h\curl \bs w_h) = (\curl \phi_h,\curl\curl_h\curl\bs w_h),
$$
and the right hand side becomes
$$(\bs f_h,\curl \bs w_h)
  = (\curl\bs g_h,\curl w_h) = (\bs g_h,\curl_h\curl\bs w_h)
$$
Denoted by $\bs \tau_h = \curl_h\curl \bs w_h\in K_h^c$. We get
$$
(\curl\bs \phi_h,\curl \bs \tau_h) = (\bs g_h,\bs \tau_h)   \qquad\text{for all }\tau_h\in K_h^c.
$$

Let 
 $\bs \phi\in  K^c$ satisfy the Maxwell equation:
$$ (\curl\bs \phi , \curl \bs \tau )= (Q_{ K}^c\bs g_h, \bs \tau) \quad \text{for all }\bs \tau_h \in K^c.$$
By Lemma \ref{lem:maxwell_app}, we have
$$ \|\curl(\bs \phi - \bs \phi_h) \|  \lesssim h\|\bs g_h\|. $$

When moving to the coarse space, the left hand side of equation \eqref{eq:solutionof_H_div} can be still simplified to $(\curl\bs \phi_H,\curl\tau_H)$. But the right hand side becomes 
$$
(\bs f_h,\curl \bs w_H)
  = (\curl\bs g_h,\curl w_H) \neq (\bs g_h, \curl _H\curl w_H).$$
We need to project $\bs g_h$ to the coarse space and arrives at the equation 
$$
(\curl\bs \phi_H,\curl\tau_H) = (P_H^c\bs g_h,\tau_H)   \qquad\text{for all }\tau_H\in K_H^c.
$$

Let 
 $\tilde{\bs \phi}\in  K^c$ satisfy the Maxwell equation:
$$ (\curl \tilde{\bs \phi}, \curl \bs \tau )= (Q_{ K}^cP_H^c\bs g_h, \bs \tau) \quad \text{for all }\bs \tau_h \in K^c.$$
By Lemma \ref{lem:maxwell_app}, we have
$$ \|\curl(\tilde{\bs \phi} - \bs \phi_H) \|  \lesssim H\|Q_{ K}^cP_H^c\bs g_h\| \leq H\|P_H^c\bs g_h\| \lesssim H\|\bs g_h\|. $$

By the triangle inequality, it remains to estimate $\|\curl(\bs \phi - \tilde{\bs \phi})\|$. We first write out the error equation for $\bs \phi - \tilde{\bs \phi}$
\begin{equation*}
  (\curl(\bs\phi - \tilde{\bs\phi}),\curl\bs\psi)  = (Q_K^c(\bs g_h - P_H^c\bs g_h) ,\bs\psi), \quad \text{ for all } \bs \psi \in K^c.
\end{equation*}
We then apply the standard duality argument. Let $\bs \zeta\in  K^c$ satisfies
$$
(\curl\bs\zeta,\curl\bs\tau)  =  (\bs\psi ,\bs\tau) \quad  \text{for all }\bs\tau\in  K^c
$$
Then
\begin{align*}
(Q_K^c(\bs g_h - P_H^c\bs g_h) ,\bs\psi) & = (\curl\bs\zeta, \curl Q_K^c(\bs g_h - P_H^c\bs g_h)) \\
&= (\curl\bs\zeta, \curl(\bs g_h - P_H^c\bs g_h))\\
  & =  (\curl(\bs\zeta - P_H^c\bs\zeta), \curl(\bs g_h - P_H^c\bs g_h)) \\
  & \lesssim  H^2\|\curl \bs\zeta\|_2 \|\curl \bs g_h\| \lesssim H\|\curl\bs \psi\| \|\bs g_h\|,
\end{align*}
which implies
$$
\|\curl(\bs\phi - \tilde{\bs \phi})\| \lesssim H\|\bs g_h\|.
$$

The estimate of $\|\curl \bs\phi_h - \curl\bs\phi_H\|$ then follows from the triangle inequality. 
%
%
\end{proof}

As a summary the the above results, we have the following theorem.
\begin{theorem}\label{them:A_1}
Condition (A.1) holds with $\alpha = \frac{1}{2}$, i.e. for any $\bs u_k\in\bs V_k$, there hold
\begin{equation}\label{eq:A_1} 
a_k^d((I-P_{k-1}) \bs u_k,\bs  u_k) \lesssim  \left(\frac{\|A_k^d\bs u_k\|^2}{\lambda_k}\right)^{\frac{1}{2}} a_k^d( \bs u_k,\bs u_k)^{\frac{1}{2}}.
\end{equation}
\end{theorem}

\subsection{Results}
According to the multigrid framework in \cite{Bramble.J;Pasciak.J;Xu.J1991}, we conclude that the variable V-cycle multigrid algorithm is a good preconditioner for the Schur complement equations \eqref{schur-complement} and \eqref{schur-complement-div}. We summarize the result in the following theorem.

\begin{theorem}
Let $V_k$ denote the operator of one V-ycle of $MG_k$ in Algorithm 2 with homogenous data, i.e., $\bs f_k = 0$.
Assume the smoothing steps $m_k$ satisfy
$$  \beta_0 m_k \leq m_{k-1} \leq \beta_1 m_k. $$
Here we assume that $\beta_0$ and $\beta_1$ are constants which are greater than one and independent of $k$. Then the condition number of $V_JA_J$ is $\mathcal O(1)$.
\end{theorem}

\begin{remark}\rm
 As noticed in \cite{Bramble.J;Pasciak.J1992}, W-cycle or two V-cycles may not be a valid preconditioner as the corresponding operator may not be positive definite. In other words, the proposed multigrid method for the Schur complement cannot be used as an iterative method but one V-cycle can be used as an effective preconditioner. \qed
\end{remark}

\section{Uniform Preconditioner}
In this section, we will show that the multigrid solver for the Schur complement equations can be used to build efficient preconditioners for the mixed formulations of vector Laplacian \eqref{matvecLap} and \eqref{matvecLap-div}. We also apply the multigrid preconditioner of the vector Laplacian to the Maxwell equation discretized as a saddle point system. We prove that the preconditioned systems have condition numbers independent of mesh parameter $h$.

\subsection{Block Diagonal Preconditioner}
It is easy to see that the inverses of the symmetric positive definite matrices $M_{v}$, $M_{e}$, $A_{h}^{c}$ and $A_{h}^{d}$ exist, which implies the existence of the operators $(\mcal L_{h}^{c})^{-1}$, $(\mcal L_{h}^{d})^{-1}$, and the block diagonal preconditioners defined as following.
\begin{definition}
We define the operator $\mcal P_{h}^{c} :  S_{h}'\times \bs U_{h}' \rightarrow S_{h}\times \bs U_{h}$ with the matrix representation
\begin{equation}
\mcal P_{h}^{c}
=
\begin{pmatrix}
 M_v^{-1} &0 \\
0& (A^{c}_{h})^{-1}
\end{pmatrix},
\end{equation}
and the operator $\mcal P_{h}^{d} :  \bs U_{h}'\times \bs V_{h}' \rightarrow \bs U_{h}\times \bs V_{h}$ with the matrix representation
\begin{equation}
\mcal P_{h}^{d}
=
\begin{pmatrix}
 M_e^{-1} &0 \\
0& (A^{d}_{h})^{-1}
\end{pmatrix}.
\end{equation}
\end{definition}

Follow the framework in~\cite{MardalWinther2004}, it suffices to prove the boundedness of operators $\mcal L_{h}^{c}$ and $\mcal L_{h}^{d}$ and their inverse in appropriate norms. 
In the sequel, to unify the notation, we use $M$ for the mass matrix and $A$ the vector Laplacian. The inverse of the mass matrix can be thought of as the matrix representation of the Riesz representation induced by the $L^2$-inner product and the inverse of $A$ is the Riesz representation of the $A$-inner product. The preconditioners $\mcal P_{h}^{c}$ and $\mcal P_{h}^{d}$ are Riesz representation of $L^2\times A$-inner product. Let $\langle \cdot, \cdot \rangle$ be the duality pair in $\mcal V_h$. We clarify the norm notations using $M$ and $A$ as follows:
\begin{itemize}
\item $\|\cdot\|_{M}$: \quad $\|\sigma_{h}\|_{M}^{2} = \langle M \sigma_{h}, \sigma_{h}\rangle$;
\smallskip
\item $\|\cdot\|_{A}$: \quad $\|u_{h}\|_{A}^{2} = \langle A_{h} u_{h}, u_{h}\rangle $;
\smallskip
\item $\|\cdot\|_{M^{-1}}$: $\|g_{h}\|_{M^{-1}}^{2} = \langle M^{-1}g_{h}, g_{h}\rangle $;
\smallskip
\item $\|\cdot\|_{A^{-1}}$: $\|f_{h}\|_{A^{-1}}^{2} = \langle A_{h}^{-1} f_{h}, f_{h}\rangle $.
\end{itemize}
\smallskip

The following lemma gives a bound of the Schur complement $BA^{-1}B^T$ similar to the corresponding result of the Stokes equation.

\begin{lemma}\label{lm:BAB}
We have the inequality
\begin{equation}\label{BAB}
\langle B(A_{h}^{c})^{-1} B^{T} \phi_{h}, \phi_{h}\rangle \leq \langle M_{v}\phi_{h}, \phi_{h}\rangle\quad \text{for all } \phi_{h} \in S_{h},
\end{equation}
\end{lemma}
\begin{proof}
Let $\bs v_{h} = (A_{h}^{c})^{-1} B^{T} \phi_{h}$. Then
\begin{align*}
\langle  B(A_{h}^{c})^{-1} B^{T} \phi_{h}, \phi_{h} \rangle =\langle  (A_{h}^{c})^{-1} B^{T} \phi_{h}, B^{T}\phi_{h}\rangle = \langle A_{h}^{c} \bs v_{h}, \bs v_{h}\rangle = \|\bs v_{h}\|_{A}^{2}.
\end{align*}
Now we identify $\bs v_{h}\in \bs V_{h}'$ by the Riesz map in the $A$-inner product, and then we have
\begin{align*}
\|\bs v_{h}\|_{A} &= \sup_{\bs u_{h}\in \bs V_{h}}\dfrac{\langle \bs v_{h},\bs u_{h}\rangle _{A}}{\|\bs u_{h}\|_{A}}
= \sup_{\bs u_{h}\in \bs V_{h}}\dfrac{\langle B^{T} \phi_{h},\bs u_{h}\rangle}{\|\bs u_{h}\|_{A}}
= \sup_{\bs u_{h}\in \bs V_{h}}\dfrac{\langle \phi_{h},B\bs u_{h}\rangle}{\|\bs u_{h}\|_{A}} \\
&\leq \sup_{\bs u_{h}\in \bs V_{h}}\dfrac{\|\phi_{h}\|_M\|B\bs u_{h}\|_{M^{-1}}}{\|\bs u_{h}\|_{A}}
\leq \|\phi_{h}\|_M.
\end{align*}
In the last step, we have used the identity \eqref{Ahc} which implies $\|B\bs u_{h}\|_{M^{-1}}\leq \|\bs u_{h}\|_{A}$.
The desired result \eqref{BAB} then follows easily.
%
\end{proof}

We present a stability result of the mixed formulation of the vector Laplacian which is different with that established in \cite{Arnold2006}.

\begin{theorem}\label{l-bound}
The operators $\mcal L_{h}^{c}, \mcal L_{h}^{d}$ and there inverse are both bounded operators:
$$
\|\mcal L_{h}^{c}\|_{{\rm L}( S_{h}\times \bs U_{h} , S_{h}'\times \bs U_{h}')},
\|\mcal L_{h}^{d}\|_{{\rm L}( \bs U_{h}\times \bs V_{h}, \bs U_{h}'\times \bs V_{h}')},
$$
are bounded  and independent of $h$ from $(\|\cdot\|_{M^{-1}}, \|\cdot\|_{A^{-1}}) \to (\|\cdot\|_{M}, \|\cdot\|_{A}),$ and
$$
 \|(\mcal L_{h}^{c})^{-1}\|_{{\rm L}( S_{h}'\times \bs U_{h}' , S_{h}\times \bs U_{h})},
\|(\mcal L_{h}^{d})^{-1}\|_{{\rm L}( \bs U_{h}'\times \bs V_{h}', \bs U_{h}\times \bs V_{h})}
$$
are bounded and independent of $h$ from $(\|\cdot\|_{M}, \|\cdot\|_{A}) \to (\|\cdot\|_{M^{-1}}, \|\cdot\|_{A^{-1}})$.
\end{theorem}
\begin{proof}
We prove the $\bs H_0(\curl)$ case below. The proof of the $\bs H_0(\div)$ case is similar.

Let $(\sigma_{h}, \bs u_{h}) \in  S_{h}\times \bs U_{h}$ and $(g_{h}, \bs f_{h}) \in S_{h}'\times \bs U_{h}'$ be given by the relation with
\begin{equation}\label{Lhc}
\mcal L_{h}^{c}
\begin{pmatrix}
\sigma_h\\
\bs u_h
\end{pmatrix}
=
\begin{pmatrix}
-M_v & B\\
B^T & C^TM_fC
\end{pmatrix}
\begin{pmatrix}
\sigma_h\\
\bs u_h
\end{pmatrix}
=
\begin{pmatrix}
g_{h}\\
\bs f_{h}
\end{pmatrix}.
\end{equation}
To prove $\|\mcal L_{h}^{c}\|_{{\rm L}( S_{h}\times \bs U_{h} , S_{h}'\times \bs U_{h}')} \lesssim 1$, it is sufficient to prove
\begin{equation}\label{boundL}
\|g_{h}\|_{M^{-1}} + \|\bs f_{h}\|_{A^{-1}} \lesssim \|\sigma_{h}\|_{M} + \|\bs u_{h}\|_{A}.
\end{equation}
From \eqref{Lhc}, we have $g_{h}= -M_{v}\sigma_{h} + B \bs u_{h}$ and $\bs f_{h} = A_{h}^{c}\bs u_{h} - B^{T}M_{v}^{-1}g_{h}$. The norm of $g_h$ is easy to bound as follows
\begin{align*}
\|g_{h}\|_{M^{-1}}^{2}\leq 2 \|M_v\sigma_{h}\|_{M^{-1}}^{2} + 2\|B\bs u_{h}\|_{M^{-1}}^{2}\leq 2\|\sigma_{h}\|_{M}^{2} + 2\|\bs u_{h}\|_{A}^{2}.
\end{align*}
To bound the norm of $\bs f_h$, we first have
\begin{align*}
\|\bs f_{h}\|_{A^{-1}}^{2}\leq 2\|B^{T} M_{v}^{-1} g_{h}\|_{A^{-1}}^{2} + 2\|A_{h}^{c} \bs u_{h}\|_{A^{-1}}^{2}  \leq 2\|B^{T} M_{v}^{-1} g_{h}\|_{A^{-1}}^{2} + 2\|\bs u_{h}\|_{A}^{2}.
\end{align*}
Let $\phi_{h}  = M_{v}^{-1} g_{h}$, by Lemma \ref{lm:BAB}, we have
\begin{align*}
\|B^{T} M_{v}^{-1} g_{h}\|_{A^{-1}}^{2}=\|B^{T} \phi_{h}\|_{A^{-1}}^{2} = \langle B (A_{h}^{c})^{-1} B^{T} \phi_{h}, \phi_{h}\rangle \leq \|\phi_{h}\|_M^2= \|g_{h}\|_{M^{-1}}^{2}.
\end{align*}
Thus we get
\begin{align*}
\|\bs f_{h}\|_{A^{-1}}^{2} \leq 2\|g_{h}\|_{M^{-1}}^{2}+ 2\|\bs u_{h}\|_{A}^{2}\leq 4\|\sigma_{h}\|_{M}^{2} + 6\|\bs u_{h}\|_{A}^{2}.
\end{align*}
Then the desired inequality \eqref{boundL} follows from the bound of $\|g_h\|_{M^{-1}}$ and $\|\bs f_h\|_{A^{-1}}$.

To prove $ \|(\mcal L_{h}^{c})^{-1}\|_{{\rm L}( S_{h}'\times \bs U_{h}' , S_{h}\times \bs U_{h})} \lesssim 1$, we need to prove
\begin{equation}\label{boundLinverse}
 \|\sigma_{h}\|_{M} + \|\bs u_{h}\|_{A}\lesssim \|g_{h}\|_{M^{-1}} + \|\bs f_{h}\|_{A^{-1}}.
\end{equation}
From \eqref{Lhc}, we have $\bs u_{h} =( A_{h}^{c})^{-1}( f_{h} + B^{T}M_{v}^{-1}g_{h})$. Then
\begin{align*}
\|\bs u_{h}\|_{A}^{2} &= \| \bs f_{h} + B^{T}M_{v}^{-1}g_{h}\|_{A^{-1}}^{2}\\
&\leq 2\|\bs f_{h}\|_{A^{-1}}^{2} + 2\|B^{T}M_{v}^{-1}g_{h}\|_{A^{-1}}^{2}\leq 2\|\bs f_{h}\|_{A^{-1}}^{2} +2\|g_{h}\|_{M^{-1}}^{2}.
\end{align*}
We also have $\sigma_{h}= M_{v}^{-1}(B \bs u_{h}-g_{h})$ and thus
\begin{align*}
\|\sigma_{h}\|_{M}^{2} =  \|B\bs  u_{h}-g_{h}\|_{M^{-1}}^{2} \leq 2\|B\bs  u_{h}\|_{M^{-1}}^{2} + 2\|g_{h}\|_{M^{-1}}^{2} \leq 2\|\bs u_{h}\|_{A}^{2} + 2\|g_{h}\|_{M^{-1}}^{2}.
\end{align*}
Combining with the bound for $\|\bs u_h\|_A$, we obtain the desirable stability \eqref{boundLinverse}.

\end{proof}

\begin{remark}
By choosing $\bs v_h = M^{-1}_eB^T\sigma _h$, we can obtain the stability 
$$
\|B^T\sigma\|_{M^{-1}} \leq \|\bs f_h\|_{M^{-1}}.
$$
\end{remark}

From Theorem \ref{l-bound}, we can conclude that the proposed preconditioners are uniformly bounded with respect to $h$.
\begin{theorem}
The $\mcal P_{h}^{c}$ and $\mcal P_{h}^{d}$ are uniform preconditioners for $\mcal L_{h}^{c}$ and $\mcal L_{h}^{d}$, respectively, i.e., the corresponding operator norms
\begin{align*}
\|\mcal P_{h}^{c}\mcal L_{h}^{c}\|_{{\rm L}( S_{h}\times \bs U_{h} , S_{h}\times \bs U_{h})} , \|(\mcal L_{h}^{c}\mcal P_{h}^{c})^{-1}\|_{{\rm L}( S_{h}\times \bs U_{h} , S_{h}\times \bs U_{h})},\\
\|\mcal P_{h}^{d}\mcal L_{h}^{d}\|_{{\rm L}( \bs U_{h}\times \bs V_{h}, \bs U_{h}\times \bs V_{h})},
\|(\mcal P_{h}^{d}\mcal L_{h}^{d})^{-1}\|_{{\rm L}( \bs U_{h}\times \bs V_{h}, \bs U_{h}\times \bs V_{h})}
\end{align*}
are bounded and independent with parameter $h$.
\end{theorem}

\subsection{Mass Lumping}
The inverse of the mass matrices $M^{-1}_v$  and $M_{e}^{-1}$ are in general dense. To be practical, the exact Schur complement can be replaced by an approximation
\begin{align}
\label{masslumping-c}
\tilde A_{h}^{c} = B^T\tilde M_{v}^{-1}B + C^TM_fC,\\
\label{masslumping-d}
\tilde A_{h}^{d} = C\tilde M_{e}^{-1}C^{T} + B^T M_tB,
\end{align}
with $\tilde M_v$ and $\tilde M_{e}$ easy-to-invert matrices, e.g., diagonal or mass lumping of $M_v$ and $M_{e}$, respectively. In this way, we actually change the $L^{2}$-inner product into a discrete $L^{2}$ inner product. We then define the adjoint operators with respect to the discrete $L^{2}$-inner product. For example, we define $\widetilde{\div}_{h} \bs w_h \in S_h$, s.t.,
\begin{equation}\label{weakgradient}
\langle\widetilde{\div}_{h} \bs w_h, v_h\rangle_h : = - ( \bs w_h, \grad v_h)\quad \text{ for all } v_h\in S_h,
\end{equation}
where $\langle\cdot,\cdot\rangle_h$ is the discrete $L^2$-inner product defined by $\tilde M_v$.


The operator and matrix formulations of the vector Laplacian $\widetilde{\mcal L_{h}^{c}} :  S_{h}\times \bs U_{h} \rightarrow S_{h}'\times \bs U_{h}'$
\begin{equation}\label{matvecLap_massLumping}
\widetilde {\mcal L_{h}^{c}}
\begin{pmatrix}
\sigma_h\\
\bs u_h
\end{pmatrix}
:=
\begin{pmatrix}
-\tilde M_v & B\\
B^T & C^TM_fC
\end{pmatrix}
\begin{pmatrix}
\sigma_h\\
\bs u_h
\end{pmatrix}
=
\begin{pmatrix}
0\\
\bs f
\end{pmatrix}.
\end{equation}
And $\widetilde{\mcal L_{h}^{d}} : \bs U_{h}\times \bs V_{h} \rightarrow \bs U_{h}'\times \bs V_{h}'$
\begin{equation}\label{matvecLap_massLumping_div}
\widetilde {\mcal L_{h}^{d}}
\begin{pmatrix}
\bs \sigma_h\\
\bs u_h
\end{pmatrix}
:=
\begin{pmatrix}
-\tilde M_e & C^{T}\\
C & B^TM_tB
\end{pmatrix}
\begin{pmatrix}
\bs \sigma_h\\
\bs u_h
\end{pmatrix}
=
\begin{pmatrix}
0\\
\bs f
\end{pmatrix}.
\end{equation}
The associated diagonal preconditioners are
\begin{equation}\label{approximatediagonal}
\widetilde{\mcal P_{h}^{c}}
=
\begin{pmatrix}
 \tilde M_v^{-1} &0 \\
0& (\tilde A^{c}_{h})^{-1}
\end{pmatrix}
\end{equation}
and
\begin{equation}
\widetilde{\mcal P_{h}^{d}}
=
\begin{pmatrix}
 \tilde M_e^{-1} &0 \\
0& (\tilde A^{d}_{h})^{-1}
\end{pmatrix}.
\end{equation}

It is not hard to see that the modification of the $L^{2}$-inner product will not bring any essential difficulty to the proof of the previous results. We can easily reproduce all the results that we have proved in the previous sections with the help of the following proposition whose proof can be found in~\cite{Chen2014a}.

\begin{proposition}\label{lem:norm_equal}
Assume that the discrete $L^{2}$ norm is equivalent to the $L^{2}$ norm. Then the norm $\|\cdot\|_{ \widetilde{A}^{c}_{h}}$ is equivalent to $\|\cdot\|_{A^{c}_{h}}$, and $\|\cdot\|_{ \widetilde{A}^{d}_{h}}$ is equivalent to $\|\cdot\|_{A^{d}_{h}}$ i.e.,
\begin{align}
\|\bs u\|_{ \widetilde{A}^{c}_{h}}\lesssim \|\bs u\|_{A^{c}_{h}} \lesssim \|\bs u\|_{ \widetilde{A}^{c}_{h}}\quad \hbox{for all} ~\bs u \in \bs U_{h};\\
\|\bs u\|_{ \widetilde{A}^{d}_{h}}\lesssim \|\bs u\|_{A^{d}_{h}} \lesssim \|\bs u\|_{ \widetilde{A}^{d}_{h}}\quad \hbox{for all} ~\bs u \in \bs V_{h}.
\end{align}
\end{proposition}

\subsection{Triangular Preconditioner}
When a diagonal mass matrix is used, we can make use of the block decomposition
\begin{equation}
\begin{pmatrix}
 -\tilde M_v & B\\
 B^T & C^TM_fC
\end{pmatrix}
\begin{pmatrix}
 I   &  \tilde M_v^{-1} B \\
0  &   I
\end{pmatrix}
=
\begin{pmatrix}
 -\tilde M_v    &   0  \\
B^T   &  \tilde A_h^c
\end{pmatrix}
\end{equation}
to obtain a triangular preconditioner.

\begin{definition}
We define the operator $\mathcal G_h^c: S_h'\times \bs U_h' \rightarrow  S_h \times \bs U_h$
\begin{equation}\label{eq:Matrix_Pre_DGS_curl}
\mathcal G_h^c = \begin{pmatrix}
I   &  \tilde M_v^{-1} B \\
0  &   I
\end{pmatrix}
\begin{pmatrix}
-\tilde M_v    &   0  \\
B^T   &  \tilde A_h^c
\end{pmatrix}
^{-1},
\end{equation}
and the operator $\mathcal G_h^d: \bs U_h' \times \bs V_h' \rightarrow  \bs U_h \times \bs V_h$
\begin{equation} \label{eq:Matrix_Pre_DGS_div}
\mathcal G_h^d = \begin{pmatrix} I   &  \tilde M_e^{-1} C^T \\ 0  &   I  \end{pmatrix}\begin{pmatrix}  -\tilde M_e    &   0  \\  C^T   &  \tilde A_h^d   \end{pmatrix}^{-1}.
\end{equation}
\end{definition}

From the definition, it is trivial to verify that $\mathcal G_h^c= \widetilde{\mcal L_{h}^{c}}^{-1}$ and $\mathcal G_h^d = \widetilde{\mcal L_{h}^{c}}^{-1}$ and thus conclude that the proposed triangular preconditioners are uniform.

\begin{theorem}
Assume $\tilde M$ is spectrally equivalent to $M$. Then the $\mcal G_{h}^{c}$ and $\mcal G_{h}^{d}$ are uniform preconditioners for $\mcal L_{h}^{c}$ and $\mcal L_{h}^{d}$, respectively, i.e., the corresponding operator norms
\begin{align*}
\|\mcal G_{h}^{c}\mcal L_{h}^{c}\|_{{\rm L}( S_{h}\times \bs U_{h} , S_{h}\times \bs U_{h})} , \|(\mcal L_{h}^{c}\mcal G_{h}^{c})^{-1}\|_{{\rm L}( S_{h}\times \bs U_{h} , S_{h}\times \bs U_{h})},\\
\|\mcal G_{h}^{d}\mcal L_{h}^{d}\|_{{\rm L}( \bs U_{h}\times \bs V_{h}, \bs U_{h}\times \bs V_{h})},
\|(\mcal G_{h}^{d}\mcal L_{h}^{d})^{-1}\|_{{\rm L}( \bs U_{h}\times \bs V_{h}, \bs U_{h}\times \bs V_{h})}
\end{align*}
are bounded and independent with parameter $h$.
\end{theorem}

In both diagonal and triangular preconditioners, to be practical, we do not compute $A^{-1}$ or $\tilde A^{-1}$. Instead we apply one and only one V-cycle multigrid for $\tilde{A}^{-1}$.

\subsection{Maxwell Equations with Divergence-Free Constraint}
We consider a prototype of Maxwell equations with divergence-free constraint
$$
\curl \curl \bs u = \bs f, \; \div \bs u = 0, \; \text{ in } \Omega, \qquad \bs u\times \bs n = 0  \text{ on } \partial \Omega.
$$
The solution $\bs u$ is approximated in the edge element space $\bs U_{h}$. The divergence-free constraint can then be understood in the weak sense, i.e., $\div_h \bs u = 0$. By introducing a Lagrangian multiplier $p\in S_{h}$, the matrix form is
\begin{equation}\label{maxwellmat}
\begin{pmatrix}
C^TM_{f}C & B^T\\
B & O
\end{pmatrix}
\begin{pmatrix}
\bs u\\
p
\end{pmatrix}
=
\begin{pmatrix}
\bs f\\
g
\end{pmatrix}.
\end{equation}
We can apply the augmented Lagrangian method~\cite{Fortin.Michel;Glowinski.Roland1983}, by adding $B^TM_{v}^{-1}B$ to the first equation, to get an equivalent matrix equation
\begin{equation}\label{maxwell}
\begin{pmatrix}
A & B^T\\
B & O
\end{pmatrix}
\begin{pmatrix}
\bs u\\
p
\end{pmatrix}
=
\begin{pmatrix}
\bs f + B^TM_{v}^{-1} g\\
g
\end{pmatrix}.
\end{equation}
Now the $(1,1)$ block $A=C^TM_{f}C + B^TM_{v}^{-1}B$ in \eqref{maxwell} is a discrete vector Laplacian and the whole system \eqref{maxwell} is in Stokes type.

We can thus use the following diagonal preconditioner.
\begin{theorem}
The following block-diagonal matrix
\begin{equation}\label{diagMax}
\begin{pmatrix}
A^{-1} &0 \\
0&  M_v^{-1}
\end{pmatrix}
\end{equation}
is a uniform preconditioner for the regularized Maxwell operator
$\begin{pmatrix}
A & B^T\\
B & O
\end{pmatrix}.$
\end{theorem}
\begin{proof}
 It suffices to prove that the Schur complement $S = BA^{-1}B^T$ is spectral equivalent to $M_v$. The inequality $(Sp,p) \leq (M_vp,p)$ for all $p\in S_h$ has been proved in Lemma \ref{lm:BAB}. To prove the inequality in the other way, it suffices to prove the inf-sup condition: there exists a constant $\beta$ independent of $h$ such that
 \begin{equation}\label{inf-sup}
\inf_{p_h\in S_h}\sup_{v_h\in U_h}\frac{\langle Bv_h, p_h\rangle }{\|v_h\|_A\|q_h\|} = \beta >0.
\end{equation}

Given $p_h \in S_h$, we solve the Poisson equation $\Delta \phi = p_h$ with homogenous Dirichlet boundary condition and let $\bs v = \grad \phi$. Then $\bs v\in \bs H_0(\curl)$ and $\div \bs v = p_h$ holds in $L^2$. We define $\bs v_h = Q_h \bs v$ where $Q_h: \bs H_0(\curl) \to \bs U_h$ is the $L^2$ projection. Then $(\div_h \bs v_h, q_h) = (\bs v_h, \grad q_h) = (\bs v, \grad q_h) = -(\div \bs v, q_h) = (p_h, q_h)$, i.e., $\div_h \bs v_h = p_h$. To control the norm of $\curl \bs v_h$, we denote $\bs v_0$ as the piecewise constant projection of $\bs v$. Then
$$
\|\curl \bs v_h\| = \|\curl (\bs v_h - \bs v_0)\|\lesssim h^{-1}\|\bs v_h - \bs v_0\| \leq \|\bs v\|_1\lesssim \|p_h\|.
$$
In the last step, we have used the $H^2$-regularity result.

In summary, given  $p_h \in S_h$, we have found a $\bs v_h \in U_h$ such that $\langle B\bs v_h, p_h\rangle = \|p_h\|^2$ while $\|\bs v_h\|_A^2 = \|\div_h \bs v_h\|^2 + \|\curl \bs v_h\|^2 \lesssim \|p_h\|^2$. Therefore the inf-sup condition \eqref{inf-sup} has been proved which implies the inequality $\langle Sp,p\rangle \geq \beta ^{2}\langle M_vp,p\rangle$.
\end{proof}
%

To design an efficient triangular preconditioner for~\eqref{maxwell}, we explore the commutator
\begin{equation}\label{AGGA}
AG = {\tilde G}A_p,
\end{equation}
where $G = M_e^{-1}B^T$ is the matrix representation of the gradient operator $S_{h}\to U_{h}$, $\hat G = B^TM_v^{-1}$ is another scaled gradient operator, and $A_p = BG$ represents the discrete Laplacian operator $S_{h}\to S_{h}$. The identity~\eqref{AGGA} is a discrete version of the following identity
\begin{equation}\label{key1}
\bs\Delta \grad = \grad \Delta,
\end{equation}
where the first $\bs \Delta$ is the vector Laplacian operator and the second $\Delta$ is the scalar Laplacian, and can be verified by noticing that $CG = \curl \grad = 0$.

\begin{figure}[hpt]
\label{fig:2D}
\subfigure[A mesh for the unit square]{
\begin{minipage}[t]{0.32\linewidth}
\centering
\includegraphics*[width=3cm]{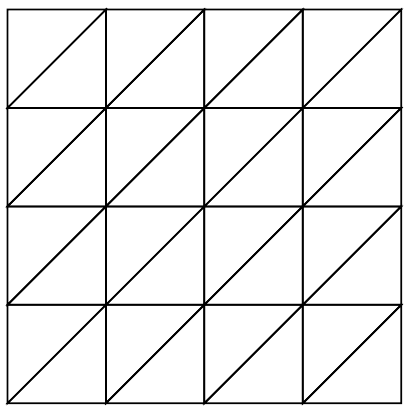}
\end{minipage}}
\subfigure[A mesh for a L-shape domain]{
\begin{minipage}[t]{0.32\linewidth}
\centering
\includegraphics*[width=3cm]{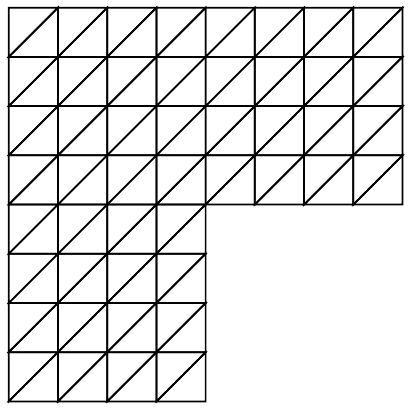}
\end{minipage}}
\subfigure[A mesh for a crack domain]{
\begin{minipage}[t]{0.35\linewidth}
\centering
\includegraphics*[width=3.25cm]{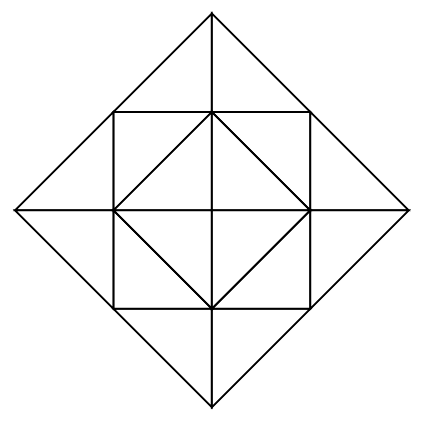}
\end{minipage}}
\caption{Meshes for Example 5.1}
\end{figure}

\begin{table}[hpt]
\footnotesize
\begin{center}
\caption{Iteration steps and CPU time of the diagonal and the triangular preconditioners for the vector Laplace equation in $\bs H_0(\curl)$ space: the square domain $(0,1)^2$.}\label{EX1-1}
\begin{tabular}{c|c|c|c|c|c}
\hline
\hline  $h$ &Dof  & Iteration (D)  & Time & Iteration (T) & Time\\
\hline
1/32           &   4,225        & 28      & 0.20 s    & 13   & 0.18s\\ \hline
1/64           &  16,641       & 28      & 0.68 s    & 14   & 0.34s \\ \hline
1/128         &    66,049     & 27      & 1.90 s      & 14   & 1.30s \\ \hline
1/256         &   263,169    & 27      & 8.80 s      & 14   & 6.80s \\ \hline
\hline
\end{tabular}
\end{center}
\end{table}

\begin{table}[hpt]
\footnotesize
\begin{center}
\caption{Iteration steps and CPU time of the diagonal and the triangular preconditioners for the lowest order discretization of the vector Laplace equation in $\bs H_0(\curl)$ space: the L-shape domain $(-1,1)^2\backslash \left \{ [0,1]\times [-1,0]\right \}$.}
\label{EX1-2}
\begin{tabular}{c|c|c|c|c|c}
\hline
\hline  $h$ &Dof  & Iteration (D)  & Time & Iteration (T) & Time\\
\hline
1/32         &   3,201       & 33      & 0.24 s  & 15  & 0.19s   \\ \hline
1/64         &  12,545      & 35      & 0.63 s  & 16  & 0.40s  \\ \hline
1/128       &    49,665    & 39      & 2.50  s   & 16  & 1.90s  \\ \hline
1/256       &   197,633   & 41      & 7.20 s    & 16  & 5.50s\\ \hline
\hline
\end{tabular}
\end{center}
\end{table}

\begin{table}[hpt]
\footnotesize
\begin{center}
\caption{Iteration steps and CPU time of the diagonal and the triangular  preconditioners for the lowest order discretization of the vector Laplace equation in $\bs H_0(\curl)$ space: the crack domain $\{|x|+|y|<1\}\backslash \{0\leq x\leq 1, y=0\}$.}
\label{EX1-3}
\begin{tabular}{c|c|c|c|c|c}
\hline
\hline  $h$ &Dof  & Iteration (D)  & Time & Iteration (T) & Time\\
\hline
1/16       &   2,145       & 34      & 0.13 s  & 15  & 0.08 s   \\ \hline
1/32       &  8,385        & 38      & 0.54 s  & 15  & 0.30 s  \\ \hline
1/64       &    33,153    & 41      & 1.60  s   & 16  & 1.00 s  \\ \hline
1/128     &   131,841   & 44      & 6.70 s    & 16  &  3.60 s\\ \hline
\hline
\end{tabular}
\end{center}
\end{table}

\begin{figure}[hpt]
\label{fig:}
\subfigure[A mesh for the unit cube]{
\begin{minipage}[t]{0.5\linewidth}
\centering
\includegraphics*[width=4cm]{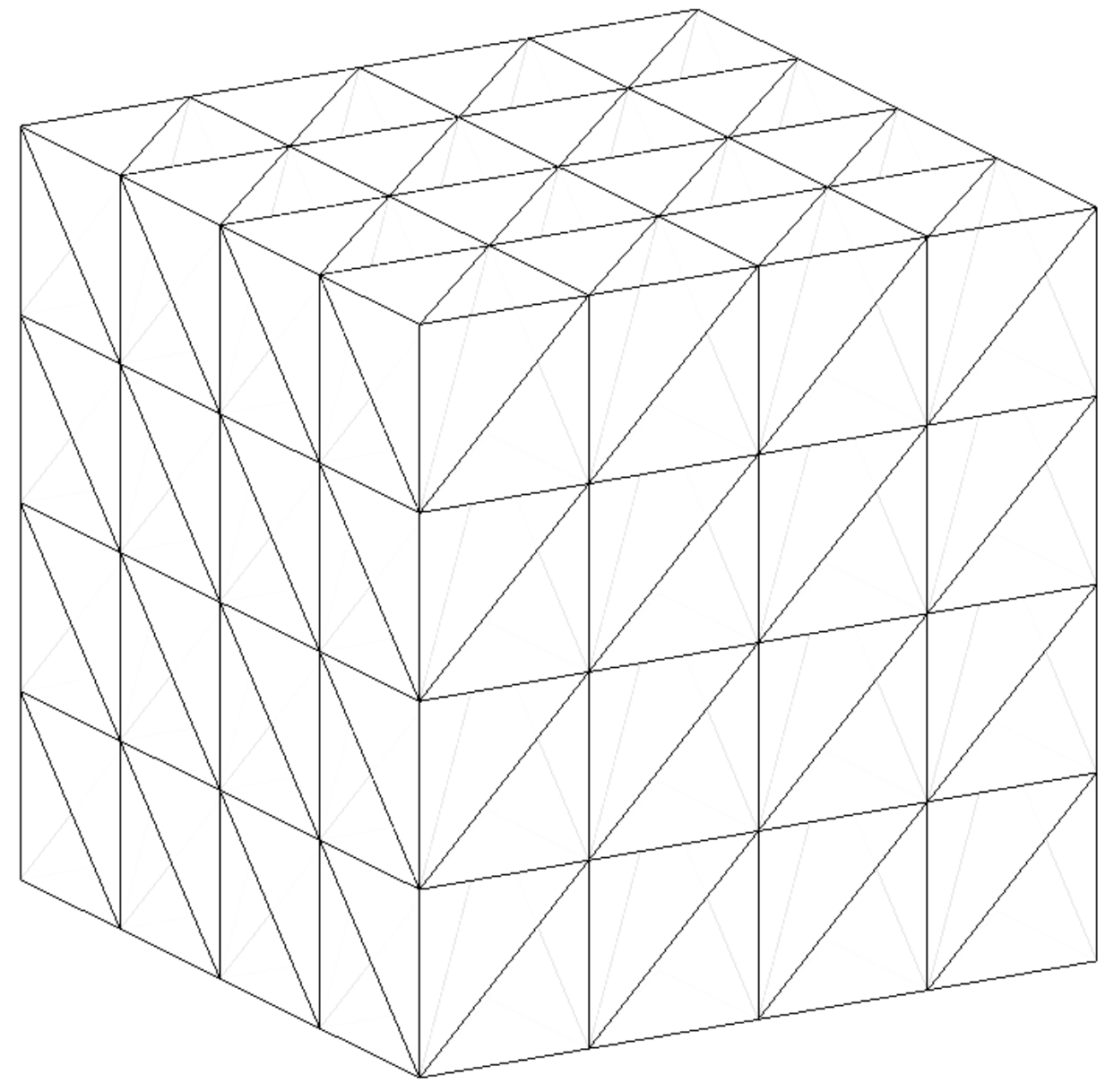}
\end{minipage}}
\subfigure[A mesh for a L-shaped domain]
{\begin{minipage}[t]{0.5\linewidth}
\centering
\includegraphics*[width=3.8cm]{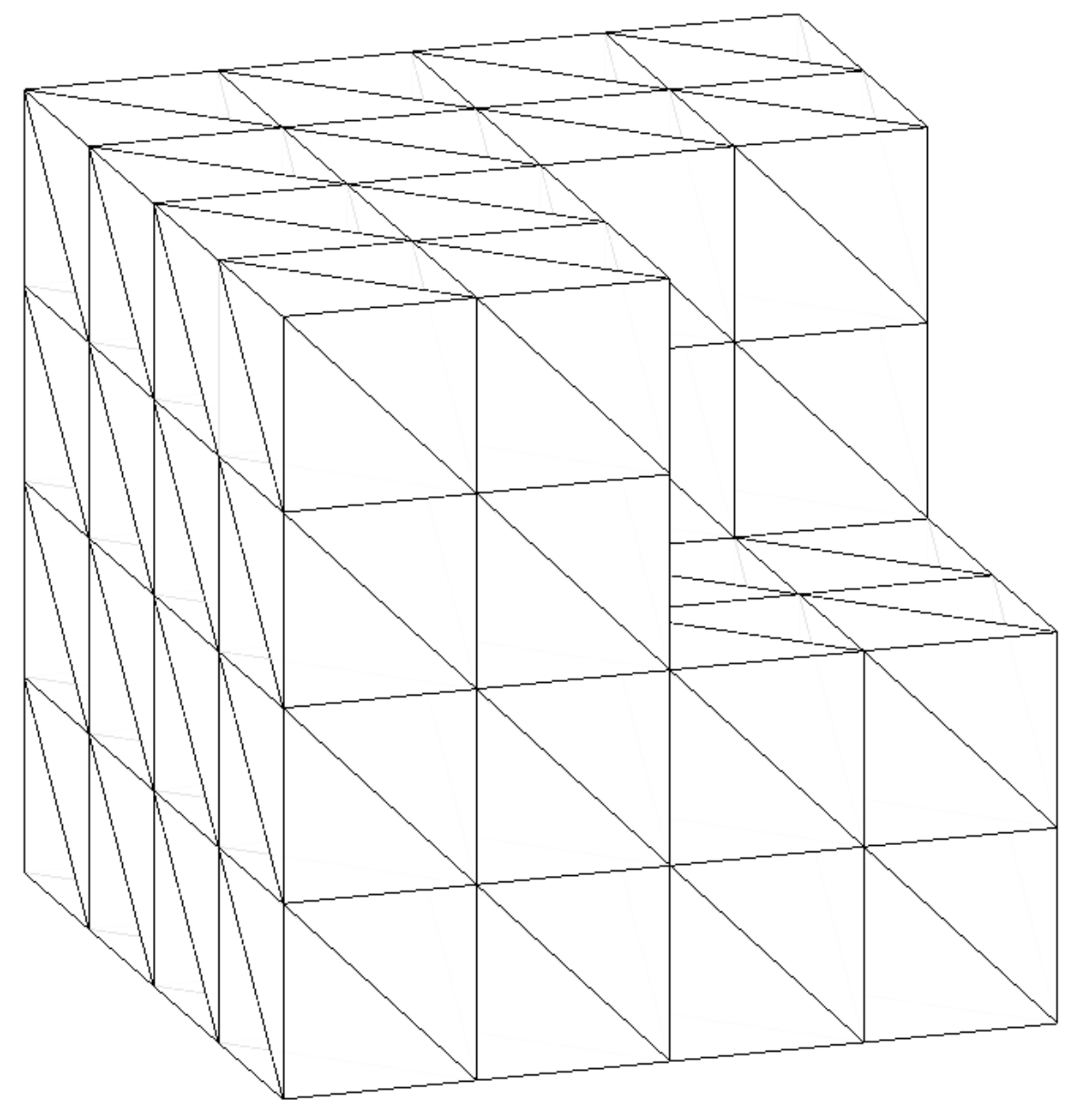}
\end{minipage}}
\caption{Meshes for Example 5.2}
\end{figure}

\begin{table}[hpt]
\footnotesize
\begin{center}
\caption{Iteration steps and CPU time of the diagonal and triangular preconditioners for the lowest order discretization of the vector Laplace equation in $\bs H_0(\curl)$ space in three dimensions: the unit cube domain.}\label{EX2-1}
\begin{tabular}{c|c|c|c|c|c}
\hline
\hline  $h$ &Dof  & Iteration (D)  & Time & Iteration (T) & Time\\
\hline
1/4       &   729       & 21      & 0.25 s  & 12  & 0.15 s   \\ \hline
1/8       &  4,913        & 29      & 0.48 s  & 16  & 0.28 s  \\ \hline
1/16       &    35,937    & 33      & 3.90  s   & 18  & 4.0 s  \\ \hline
1/32     &   274,625   & 33      & 40 s    & 19  &  27 s\\ \hline
\hline
\end{tabular}
\end{center}
\end{table}

\begin{table}[ht]
\footnotesize
\begin{center}
\caption{Iteration steps and CPU time of the diagonal and triangular preconditioners for the lowest order discretization of the vector Laplace equation in $\bs H_0(\curl)$ space in three dimensions: L-shape domain $(-1,1)^3\backslash \left \{(-1,0)\times (0,1)\times (0,1)\right \}$.}\label{EX2-2}
\begin{tabular}{c|c|c|c|c|c}
\hline
\hline  $h$ &Dof  & Iteration (D)  & Time & Iteration (T) & Time\\
\hline
1/2       &   665       & 20      & 0.03 s  & 12  & 0.06 s   \\ \hline
1/4       &  4,401        & 34      & 0.54 s  & 16  & 0.37 s  \\ \hline
1/8       &    31,841    & 42      & 5.50  s   & 20  & 3.60 s  \\ \hline
1/16     &   241,857   & 48      & 48 s    & 23  &  33 s\\ \hline
\hline
\end{tabular}
\end{center}
\end{table}

\begin{table}[hpt]
\footnotesize
\begin{center}
\caption{Iteration steps and CPU time of the diagonal and triangular preconditioners for the lowest order discretization of Maxwell equations in the saddle point form in three dimensions: the unit cube domain.}
\label{EX3-1}
\begin{tabular}{c|c|c|c|c|c}
\hline
\hline  $h$ &Dof  & Iteration (D)  & Time & Iteration (T) & Time\\
\hline
1/4       &   729       & 21      & 0.40 s  & 12  & 0.80 s   \\ \hline
1/8       &  4,913        & 27      & 1.3 s  & 16  & 1.3 s  \\ \hline
1/16       &    35,937    & 31      & 4.30  s   & 18  & 4.8 s  \\ \hline
1/32     &   274,625   & 31      & 40 s    & 19  &  39 s\\ \hline
\hline
\end{tabular}
\end{center}
\end{table}

\begin{table}[hpt]
\footnotesize
\begin{center}
\caption{Iteration steps and CPU time of the diagonal and triangular preconditioners for the lowest order discretization of Maxwell equations in the saddle point form in three dimensions: L-shape domain $(-1,1)^3\backslash \left \{(-1,0)\times (0,1)\times (0,1)\right \}$.}
\label{EX3-2}
\begin{tabular}{c|c|c|c|c|c}
\hline
\hline  $h$ &Dof  & Iteration (D)  & Time & Iteration (T) & Time\\
\hline
1/2       &   665       & 20      & 0.47 s  & 10  & 0.68 s   \\ \hline
1/4       &  4,401        & 28      & 0.58 s  & 14  & 1.10 s  \\ \hline
1/8       &    31,841    & 34      & 5.70  s   & 17  & 4.00 s  \\ \hline
1/16     &   241,857   & 37      & 40 s    & 19  &  38 s\\ \hline
\hline
\end{tabular}
\end{center}
\end{table}

With~\eqref{AGGA}, we have the following block factorization
\begin{equation}\label{maxwellfac}
\begin{pmatrix}
A & B^T\\
B & O
\end{pmatrix}
\begin{pmatrix}
I & G\\
O & -M_v^{-1}A_p
\end{pmatrix}
=
\begin{pmatrix}
A & O\\
B & A_p
\end{pmatrix}.
\end{equation}
When $S_h$ is the linear ($P_1$) element, $M_v^{-1}$ can be approximated accurately by using the mass lumping of the $P_1$ element. Therefore we can easily solve~\eqref{maxwellmat} by inverting two Laplacian operators: one is a vector Laplacian of the edge element and another is a scalar Laplacian for the $P_1$ element. In general $M_v^{-1}$ will be replaced by a sparse approximation $\tilde M_{v}^{-1}$ and~\eqref{maxwellfac} can be used to construct effective block-triangular preconditioners:
\begin{equation}\label{triMax}
\begin{pmatrix}
I & G\\
O & -\tilde M_v^{-1}A_p
\end{pmatrix}
\begin{pmatrix}
\tilde A & O\\
B & A_p
\end{pmatrix}^{-1}.
\end{equation}
Again in practice, $\tilde A^{-1}$ and $A_p^{-1}$ will be replaced by one multigrid V-cycle for the vector Laplacian or scalar Laplacian, respectively.

\section{Numerical Examples}
In this section, we will show  the efficiency  and the robustness of the proposed diagonal and triangular preconditioners. We perform the numerical experiments using the $i$FEM package \cite{Chen.L2008c}.

\begin{example}[Two Dimensional Vector Laplacian using Edge Elements]\rm
We first  consider the mixed system \eqref{matvecLap} arising from the lowest order discretization of the vector Laplace equation in $\bs H_0(\curl)$ space.

We consider three domains in two dimensions: the unit square $(0,1)^2$, the L-shape domain $(-1,1)^2\backslash \left \{ [0,1]\times [-1,0]\right \}$, and the crack domain $\{|x|+|y|<1\}\backslash \{0\leq x\leq 1, y=0\}$; see Fig. \ref{fig:2D}.
\end{example}

We use the diagonal preconditioner \eqref{approximatediagonal} in the MINRES method and the triangular preconditioner \eqref{eq:Matrix_Pre_DGS_curl} in GMRES (with the restart step $20$) to solve \eqref{matvecLap}.
In these preconditioners, one and only one variable V-cycle is used for approximating $\tilde{A}^{-1}$. In the variable V-cycle, we chose $m_J = 2$ and $m_{k} = \lceil 1.5^{J - k}m_J \rceil$ for $k=J,\ldots, 1$. We stop the Krylov space iteration when the relative residual is less than or equal to $10^{-8}$. Iteration steps and CPU time are summarized in Table \ref{EX1-1}, \ref{EX1-2}, and \ref{EX1-3}.

\begin{example}[Three Dimensional Vector Laplacian using Edge Elements]\rm
We then consider the three dimensional case. Still consider the lowest order discretization of the vector Laplace equation in $\bs H_0(\curl)$ space. We use almost the same setting except $m_J = 3$ for which the performance is more robust.

We consider two domains. One is the unit cube $(0,1)^3$ for which the full regularity assumption holds and another is a L-shape domain $(-1,1)^3\backslash \left \{(-1,0)\times (0,1)\times (0,1)\right \}$ which violates the full regularity assumption. Iteration steps and CPU time are summarized in Table \ref{EX2-1} and \ref{EX2-2}.
\end{example}

Based on these tables, we present some discussion on our preconditioners.
\begin{enumerate}
 \item Both diagonal and triangular preconditioners perform very well. The triangular one is more robust and efficient.

 \item The diagonal preconditioner is more sensitive to the elliptic regularity result as the iteration steps are slowly increased, which is more evident in the three dimensional case; see the third column of Table \ref{EX2-1} and \ref{EX2-2}. For general domains, the $\bs H_0(\curl)\cap \bs H(\div)$ is a strict subspace of $\bs H^1$ and thus the approximation property may fail. On the other hand, the numerical effectiveness even in the partial regularity cases is probably due to the fact that the full regularity of elliptic equations always holds in the interior of the domain. Additional smoothing for near boundary region might compensate the loss of full regularity.

 \item Only the lowest order element is tested while our theory assumes the finite element space should contain the full linear polynomial to ensure the approximation property. This violation may also contribute to the slow increase of the iteration steps. We do not test the second type of edge element due to the complication of the prolongation operators. The lowest order edge element is the most popular edge element. For high order edge elements, we prefer to use the V-cycle for the lowest order element plus additional Gauss-Seidel smoothers in the finest level to construct preconditioners.
\end{enumerate}

\begin{example}[Three dimensional Maxwell equations with divergent-free constraint] \rm
We consider the lowest order discretization of Maxwell equations in the saddle point form \eqref{maxwellmat} and solve the regularized formulation \eqref{maxwell}. We test the block-diagonal preconditioner \eqref{diagMax} and triangular preconditioner \eqref{triMax}. We use the same setting as in Example 5.2 and report the iteration steps and corresponding CPU time in Table \ref{EX3-1} and \ref{EX3-2}.

From these results, we conclude our block-diagonal and block-triangular preconditioners works pretty well for the Maxwell equations discretized in the saddle point form. The iteration steps may increase but very slowly. Although the block-triangular preconditioner requires less iteration steps, the computational time is almost the same. This is due to the fact, now for the $(2,2)$ block, the block-triangular preconditioners requires a V-cycle for the scalar Laplacian while in the block-diagonal preconditioner it is only a diagonal approximation of the mass matrix.
\end{example}

%
%
%


\end{document}

%% file: mysetting.tex
\usepackage{times,amsmath,amsbsy,amssymb,amscd,mathrsfs}
\usepackage{slashbox}\usepackage{graphicx,subfigure,epstopdf,wrapfig,chemarrow}
\usepackage{multicol,multirow}
\usepackage{mathtools}
\usepackage[usenames,dvipsnames,svgnames,table]{xcolor}

\usepackage{comment,enumerate,multicol,xspace}

  \newcounter{mnote}
  \let\oldmarginpar\marginpar
    \renewcommand\marginpar[1]{\-\oldmarginpar[\raggedleft\footnotesize #1]%
    {\raggedright\footnotesize #1}}

\newtheorem{theorem}{Theorem}[section]
\newtheorem{lemma}[theorem]{Lemma}

\newtheorem{proposition}[theorem]{Proposition}
\newtheorem{definition}[theorem]{Definition}
\newtheorem{example}[theorem]{Example}

\newtheorem{remark}[theorem]{Remark}

\newcommand{\mbb}{\mathbb}

\newcommand{\bs}{\boldsymbol}
\newcommand{\mcal}{\mathcal}
\newcommand{\curl}{{\rm curl\,}}
\renewcommand{\div}{\operatorname{div}}
\newcommand{\grad}{{\rm grad\,}}